\documentclass{article}
\usepackage[greek, english]{babel}
\usepackage{amsmath}
\usepackage{amssymb}
\usepackage{eufrak}
\usepackage{eurosym}
\usepackage{epsf}
\usepackage{epsfig}

\newfont{\gothic}{eufm10}

   \def\C{{\CC}} \def\CC{{\mathbb{C}}}
\def\Z{{\mathbb{Z}}}                   \def\R{{\RR}} \def\Q{{\mathbb{Q}}}
\def\RR{{\mathbb{R}}}        \def\N{{\mathbb{N}}}

        \newtheorem{theorem}{Theorem}[section]
\newtheorem{lemma}[theorem]{Lemma}
\newtheorem{proposition}[theorem]{Proposition}

\newtheorem{definition1}[theorem]{Definition}
\newenvironment{definition}{\begin{definition1}\rm}{\end{definition1}}
\newenvironment{proof}{\addvspace\baselineskip\noindent{\it
Proof:}\quad}{\hspace*{\fill}         $\Box$\par\addvspace\baselineskip}

\newtheorem{remark1}[theorem]{Remark}

\newtheorem{example1}[theorem]{Example}
\newenvironment{example}{\begin{example1}\rm}{\end{example1}}

\def\barray{\begin{eqnarray*}}             \def\earray{\end{eqnarray*}}
\def\beq{\begin{equation}} \def\eeq{\end{equation}}

\makeatletter \title{On the destruction of minimal foliations}
\author{Blaz Mramor\thanks{Department of Mathematics, VU University Amsterdam, The Netherlands, {\tt b.mramor@vu.nl}.} \ and Bob Rink\thanks{Department of Mathematics, VU University Amsterdam, The Netherlands, {\tt b.w.rink@vu.nl}.}\ . }
\begin{document}  \hyphenation{boun-da-ry mo-no-dro-my sin-gu-la-ri-ty ma-ni-fold ma-ni-folds re-fe-rence se-cond se-ve-ral dia-go-na-lised con-ti-nuous thres-hold re-sul-ting fi-nite-di-men-sio-nal ap-proxi-ma-tion pro-per-ties ri-go-rous mo-dels mo-no-to-ni-ci-ty pe-ri-o-di-ci-ties mi-ni-mi-zer mi-ni-mi-zers know-ledge ap-proxi-mate pro-per-ty}
\newcommand{\X}{\mathbb{X}}

\newcommand{\p}{\partial}
\maketitle
\noindent

\abstract{\noindent Monotone variational recurrence relations such as the Frenkel-Kontorova lattice, arise in solid state physics, conservative lattice dynamics and as Hamiltonian twist maps. \\ 
\indent For such recurrence relations, Aubry-Mather theory guarantees the existence of solutions of every rotation number $\omega \in \R$. They are the action minimizers that constitute the Aubry-Mather set. When $\omega$ is irrational, the Aubry-Mather set is either connected or a Cantor set. A connected Aubry-Mather set is called a {\it minimal foliation}. In the case of twist maps, it  describes an invariant circle, while in solid state physics it corresponds to a continuum of ground states. A Cantor Aubry-Mather set is called a {\it minimal lamination}. \\
\indent In this paper we prove that when the rotation number of a minimal foliation is either rational or easy to approximate by rational numbers, then the foliation can be destroyed into a lamination by an arbitrarily small smooth perturbation of the recurrence relation. This generalizes a theorem of Mather for twist maps to general recurrence relations.}

\section{Introduction}
In this paper, we look for
real-valued sequences $x:\Z\to \R$ that satisfy a monotone variational recurrence relation of the form
\begin{align}\label{recrel}
\sum_{j\in \Z} \p_iS_j(x) = 0\ \mbox{for all} \ i\in \Z\ .
\end{align}
In Section \ref{examplessubsection} we will explain how such recurrence relations arise in solid state physics, in the study of lattice mechanical systems and in the theory of Hamiltonian twist maps. Concrete examples to have in mind are generalized Frenkel-Kontorova crystal models with interactions of finite range, such as the one described by the recurrence relation
\begin{align}\label{FKnextstat}
x_{i+2} + x_{i+1}  - 4x_i + x_{i-1} +x_{i-2} -V'(x_i)  = 0 \ \mbox{for all}  \ i \in \Z \ .
\end{align}
Equation (\ref{FKnextstat}) defines the equilibrium states of a crystal in which the atoms are attracted by their nearest and next-nearest neighbors and also feel the influence of a conservative periodic background force.\\ 
\indent
We are interested in minimal foliations for (\ref{recrel}). A minimal foliation is a certain continuous and well-ordered family of solutions. In the context of a Hamiltonian twist map, it describes an invariant circle, while in the setting of solid state physics and lattice mechanics, it corresponds to a continuum of equilibrium states. \\
\indent The main result of this paper is a converse KAM theorem for minimal foliations. It concerns the case that the rotation number of the foliation is easy to approximate by rational numbers, for example when this rotation number is a Liouville number. We show that the foliation can then be destroyed by changing the recurrence relation by an arbitrarily small $C^{\infty}$ perturbation. This means that minimal foliations are unstable under small perturbations unless their rotation number is very irrational. \\
\indent This destruction result is a generalization of a result obtained by Mather \cite{Matherdestruction} for twist maps. We present a quite different proof though, that works for general recurrence relations of the form (\ref{recrel}).

\subsection{Requirements on the potentials} \label{problemsetup}
Let us clarify the meaning of (\ref{recrel}). First of all, the functions $S_j: \R^{\Z}\to \R$ in (\ref{recrel}) are defined for all $j\in \Z$ and assign a real value to every sequence. We think of $S_j(x)$ as the ``local energy'' of the configuration $x$ at lattice site $j$ and hence the $S_j$ will be called {\it local potentials}. By $\p_iS_j:=\frac{\partial S_j}{\partial x_i}$ we denote the partial derivative of $S_j$ with respect to the $i$-th coordinate of $x$. \\
\indent In order for (\ref{recrel}) to be well-defined and have interesting solutions, we shall require that the $S_j$ satisfy conditions {\bf A}-{\bf E} below.  
\begin{itemize}
\item[{\bf A.}] Each function $S_j$ is of finite range. This means that there is an integer $r$ so that $S_j(x)$ depends only on the values of 
$$x_{j-r}, x_{j-r+1}, \ldots, x_{j+r-1}\ \mbox{and}\ x_{j+r}\ .$$ 
A formal way of expressing this, is that there exist functions $s_j: \R^{\{j-r, \ldots, j+r\}}\to \R$ such that $S_j(x)=s_j(x|_{\{j-r,\ldots, j+r\}})$. We require that each $S_j$ is twice continuously differentiable. In particular, condition {\bf A} guarantees that each sum in (\ref{recrel}) is finite.
\item[{\bf B.}] The $S_j$ are invariant under the $\Z\times\Z-$action of shifting sequences over integers. This action is determined by {\it shift maps} $\tau_{k,l}: \R^{\Z}\to \R^{\Z}$, for integers $k$ and $l$, defined by
\begin{align}\label{definitiontau}
(\tau_{k,l}x)_i := x_{i-k}+l\ .
\end{align}
The word ``shift-map'' refers to the fact that the graph of $\tau_{k,l}x$, viewed as a subset of $\mathbb{Z}\times \mathbb{R}$, is obtained by shifting the graph of $x$ over the integer vector $(k,l)$. 
The required invariance property now is that 
\begin{align}\label{tauinvariance}
S_j(x)=S_{j+k}(\tau_{k,l}x)\ \mbox{for all}\ j, k, l\in \Z\ .
\end{align}
This shift-invariance expresses the spatial homogeneity of the local potentials. In particular, once one of the $S_j$ is given, for instance $S_0$, then all the others are determined by (\ref{tauinvariance}).  
 \item[{\bf C.}] The functions $S_j$ are coercive, i.e. they grow at infinity. More precisely, 
 $$\lim_{|x_k-x_j|\to \infty}S_j(x)=\infty\ \mbox{if}\ |k-j|=1\ .$$
This condition expresses that every function $x\mapsto S_j(x)$ is as coercive as it can be under the periodicity condition $S_j(\tau_{0,1}x)=S_j(x)$. 
\item[{\bf D.}] The $S_j$ are {\it monotone} in the sense that their mixed derivatives have a sign:
$$\partial_{i,k}S_j  \leq 0 \ \mbox{for all} \ j \ \mbox{and all} \ i\neq k, \mbox{while} \ \partial_{j,k}S_j  < 0 \ \mbox{for all} \ |j-k|=1 \ .$$
Condition {\bf D} is also called the {\it twist condition} or {\it ferromagnetic condition}. It implies that (\ref{recrel}) is a monotone recurrence relation, in the sense that the derivative of the left hand side of (\ref{recrel}) with respect to any of the $x_k$ with $k\neq i$, is non-positive. 
 \item[{\bf E.}] For technical reasons alone, we require that the $S_j$ have bounded derivatives, i.e. there is a constant $C>0$ so that
$$|\partial_iS_j(x)| \leq C\ \mbox{for all}\ i,j\ \mbox{and} \ |\partial_{i, k}S_j(x)| \leq C\ \mbox{for all} \ i, j, k \ \mbox{and uniformly in} \ x. $$
\end{itemize}

\subsection{Examples}\label{examplessubsection}
The type of recurrence relation (\ref{recrel}) is found in many applications, most notably in the theory of Hamiltonian twist maps of the cylinder and in the study of equilibrium states of ferromagnetic crystals. In this section, we will briefly explain these well-known facts. \\
\indent Hamiltonian twist maps of the cylinder are used to describe convex billiards \cite{Tabachnikov} and they also arise generically as Poincar\'e maps of two degree of freedom Hamiltonian systems near elliptic equilibria \cite{MatherForni}. We would like to remind the reader that a
{\it Hamiltonian map} of the cylinder is a map 
$$T: \R/\Z \times \R \to \R/\Z \times \R $$ 
for which there exists a time-periodic real-valued Hamiltonian function 
$$H = H(x \!\!\!\! \mod \Z, y, t \!\!\!\! \mod \Z) \ \mbox{on} \ \R/\Z \times \R \times \R/\Z$$ 
so that $T$ equals the time-$1$ flow or ``Poincar\'e map'' of the non-autonomous canonical Hamiltonian differential equations 
$$\frac{dx}{dt} = \frac{\p H(x, y, t)}{\p y}, \ \frac{dy}{dt} = - \frac{\p H(x, y, t)}{\p x}\ .$$
Hamiltonian maps are sometimes also called ``exact symplectic maps''.\\
\indent Putting it loosely, $T$ is called a {\it twist map} if it ``twists'' the cylinder $\R/\Z \times \Z$ so much that  $T$ is globally equivalent to a recurrence relation. Or, putting it more precisely, if the map $(x,y)\mapsto (x, X)=(x, T_1(x,y))$ is a global orientation preserving diffeomorphism of $\R^2$. \\
\indent One can show, see \cite{gole01} or \cite{MramorRink1}, that a Hamiltonian twist map admits a so-called {\it generating function} $S = S(x,X)$ from $\R\times \R$ into $\R$ with the property that
a sequence 
$$\{\ldots, (x_{-1} \!\!\!\! \mod \Z, y_{-1}), (x_0  \!\!\!\! \mod \Z, y_0), (x_1  \!\!\!\! \mod \Z, y_1), \ldots \} \subset (\R/\Z \times \R)^{\Z}$$
of points in the cylinder, is an orbit of $T$ if and only if 
\begin{align}\label{twistmaprecurrence}
\p_X S(x_{i-1}, x_i) + \p_xS(x_i, x_{i+1})=0\ \mbox{and} \ y_i = -\p_xS(x_{i}, x_{i+1})\ \mbox{for all} \ i\in \Z\ .
\end{align}
The first equality in (\ref{twistmaprecurrence}) defines a recurrence relation in the $x_i$. It is precisely of the form (\ref{recrel}) if one sets $S_j(x):=\frac{1}{2}S(x_{j-1}, x_j) + \frac{1}{2}S(x_j, x_{j+1})$. It turns out that these $S_j$ satisfy requirements {\bf A}-{\bf E}. This explains how (\ref{recrel}) occurs in the theory of Hamiltonian twist maps.
\\
\indent A famous  example of a Hamiltonian twist map is Chirikov's standard map, given by 
$$T_V(x\!\!\!\! \mod \! \Z, y) := \left(x+y + V'(x) \!\!\!\! \mod \! \Z, y + V'(x) \right) ,$$
where $V:\R\to \R$ is some smooth function that satisfies $V(\xi+1)=V(\xi)$ for all $\xi\in \R$. One can compute that Chirikov's standard map is equivalent to $y_i=x_{i+1}-x_i-V'(x_i)$ for all $i\in \Z$, together with the recurrence relation
\begin{align}\label{FKrecrel}
x_{i+1} - 2x_i + x_{i-1} - V'(x_i)=0\ \mbox{for all}  \ i \in \Z \ .
\end{align}
Indeed, (\ref{FKrecrel}) is of the form (\ref{recrel}) for
\begin{align}\label{FKpotentials}
S_j(x) = \frac{1}{4}(x_{j}-x_{j-1})^2 + \frac{1}{4}(x_{j+1}-x_j)^2 + V(x_j)\ .
\end{align}
Interestingly enough, the recurrence relation (\ref{FKrecrel}) does not only describe the orbits of Chirikov's standard map, but also the stationary solutions of the {\it Frenkel-Kontorova lattice} mechanical system or {\it Frenkel-Kontorova crystal} model
\begin{align}\label{FKdynamic}
m_i\frac{d^2x_i}{dt^2} = x_{i+1}  - 2x_i + x_{i-1} -V'(x_i)\ \mbox{for all}  \ i \in \Z \ . 
\end{align}
The Newtonian equations of motion (\ref{FKdynamic}) model an infinite one-dimensional array of particles that attract their nearest neighbors and moreover feel the influence of a periodic background potential $V=V(\xi)$. In the special case that $V(\xi)=\frac{1}{2\pi}\sin (2\pi\xi)$, equations (\ref{FKdynamic}) describe the famous {\it sine-Gordon lattice} of coupled pendula. The variable $2\pi x_j$ then has the interpretation of the angle of the $j$-th pendulum. \\
\indent Thus, we observe that the equilibrium states of certain lattice mechanical systems or ferromagnetic crystals can be characterized as the solutions to an equation of the form (\ref{recrel}). Of course, not in all of these lattice systems do the particles interact only with their nearest neighbors. This is the case, for instance, for the ``next-nearest-neighbor'' lattice system
\begin{align}\label{FKnextnearest}
m_i \frac{d^2x_i}{dt^2} = x_{i+2} + x_{i+1}  - 4x_i + x_{i-1} +x_{i-2} -V'(x_i) \ \mbox{for all}  \ i \in \Z \ .
\end{align}
The stationary solutions of these equations are given by the higher order recurrence relation (\ref{FKnextstat}), which is indeed of the form (\ref{recrel}) if one chooses
$$S_j(x) = \sum_{|k-j|\leq 2} \frac{1}{4}(x_{k}-x_j)^2+ V(x_j) \ \mbox{for all} \ j\in \Z\ .$$
It is important to note that the recurrence relation (\ref{FKnextstat}) is not equivalent to a twist map of the cylinder. On the contrary, the equilibrium points of (\ref{FKnextnearest}) are the orbits of the $4$-dimensional map $(x_{i-2}, x_{i-1}, x_i, x_{i+1})\mapsto (x_{i-1}, x_i, x_{i+1}, x_{i+2})$, where $x_{i+2}:=V'(x_i)-x_{i-2}-x_{i-1}+4x_i-x_{i+1}$.\\
\indent We would like to stress that the results in this paper, and in particular the destruction results announced before, are also valid for such higher order recurrence relations. They therefore form a genuine generalization of certain destruction results for twist maps obtained in \cite{Matherdestruction}. In particular, our proof is different.
\\
\indent 
Before we can describe these destruction results, we need to review some classical facts from Aubry-Mather theory. In order not to overload the reader at this point, we will discuss some more classical theory later, in Section \ref{moreclassicalstuff}. For a much more complete overview of Aubry-Mather theory, we refer to \cite{MatherForni} or \cite{MramorRink1}.

\subsection{A variational principle}
One can remark that the expression at the left hand side of (\ref{recrel}) can be thought of as $\p_iW(x)=\frac{\p W(x)}{\p x_i}$,
where $W(x)$ is the formal (and generally divergent) sum 
$$W(x)=\sum_{j\in \Z} S_j(x)\ .$$
Thus, the solutions to (\ref{recrel}) can be viewed as the formal stationary points of $W(x)$. This inspires us to introduce a special type of solutions to (\ref{recrel}), which can be thought of as the formal minimizers of $W(x)$:
\begin{definition}\label{globalmin}
A sequence $x:\mathbb{Z}\to\mathbb{R}$ is called a {\it minimizer} or {\it global minimizer} or {\it ground state} for the potentials $S_j$ if for every $y:\mathbb{Z}\to\mathbb{R}$ with finite support,
\begin{align}\label{minimizingproperty}
W(x+y)-W(x):=\sum_{j\in\Z} \left( S_j(x+y)-S_j(x) \right) \geq 0 \ .
\end{align}
\end{definition}
Note that by virtue of condition {\bf A} and the assumption that $y$ has finite support, the sum in (\ref{minimizingproperty}) is finite.
It is clear that every minimizer solves (\ref{recrel}). Indeed, applying inequality (\ref{minimizingproperty}) for the $y$ defined by $y_i:=\varepsilon$ and $y_j:=0$ for all $j\neq i$, and subsequently differentiating with respect to $\varepsilon$, one recovers (\ref{recrel}). 
 
 \subsection{The set of minimizers}\label{introsectiononminimizers}
A special type of global minimizers is easy to find, as we will see in Section \ref{classicalresults}. These are the {\it periodic minimizers} that live in one of the finite-dimensional spaces
\begin{align}\label{Xpqdef}
\X_{p,q}:=\{x:\Z\to \R\ | \ \tau_{p,q}x=x\} \ \mbox{with}\ p \geq 1\ \mbox{and} \ q \ \mbox{integers}.
\end{align}
A first central result of Aubry-Mather theory is a consequence of ``Aubry's lemma'' and it concerns these periodic minimizers:
\begin{center}
{\it For every $(p,q)\in \N\times \Z$, the collection of periodic minimizers in $\X_{p,q}$ is nonempty, closed, strictly ordered and shift-invariant.}
\end{center}
We remark that ``closed'' means that when $x^{p,q}_1, x^{p,q}_2, \ldots \in \X_{p,q}$ are periodic minimizers and the pointwise limit $x^{p,q}_{\infty}=\lim_{n\to\infty}x^{p,q}_n$ exists, then also $x^{p,q}_{\infty}\in \X_{p,q}$ is a periodic minimizer. Secondly, one calls two different sequences $x, y:\Z\to\R$ ``strictly ordered'' if either $x_i<y_i$ for all $i\in \Z$ or $x_i>y_i$ for all $i\in \Z$. Thirdly, shift-invariance means that whenever $x$ is a periodic minimizer, then so are its shifts $\tau_{k,l}x$ for all integers $k,l$, see the definition in (\ref{definitiontau}).\\
\indent To distinguish periodic minimizers from non-periodic ones, one introduces the rotation number: 
\begin{definition}\label{rotnumberdef}
Let $x:\mathbb{Z}\to\mathbb{R}$ be a sequence. We say that $x$ has {\it rotation number} $\omega \in \R$ if the limit
$$\lim_{n\to \pm \infty} \frac{x_{n}}{n} \ \mbox{exists and is equal to}\  \omega\ .$$
\end{definition}
It is clear that when $x\in \X_{p,q}$, then its rotation number is $\omega=\frac{q}{p}$. But minimizers of irrational rotation number also exist. This is another principal result in Aubry-Mather theory, the proof of which can be found in \cite{bangert87}:
 \begin{center}
 {\it For every $\omega\in\R \backslash \Q$ there exists a unique nonempty, closed, strictly ordered, shift-invariant and minimal collection $\mathcal{M}^{\omega}\subset \R^{\Z}$ of minimizers of (\ref{recrel}) with rotation number $\omega$.}
 \end{center}
In the above, ``minimality'' means that $\mathcal{M}^{\omega}$ does not contain a nonempty proper subset that is also closed and shift-invariant.\\
\indent The collection $\mathcal{M}^{\omega}$ is called the {\it Aubry-Mather set} of rotation number $\omega$ and its elements are constructed as pointwise limits of periodic minimizers. This will be explained in some detail in Section \ref{classicalresults}. \\
\indent It is well-known that any nonempty, closed, strictly ordered, shift-invariant and minimal subset of $\R^{\Z}$ is either topologically connected or a Cantor set. Hence the third main result of Aubry-Mather theory:
\begin{center}
 {\it For $\omega\in\R \backslash \Q$, the Aubry-Mather set $\mathcal{M}^{\omega}$ is either topologically connected or a Cantor set.}
 \end{center}
\indent Moser \cite{moser86}, \cite{moser89} called a topologically connected, strictly ordered, shift-invariant family of solutions to (\ref{recrel}) a {\it minimal foliation} - he proved that such a family of solutions must consist of only minimizers. Thus, a connected Aubry-Mather set is an example of a minimal foliation. A Cantor Aubry-Mather set is then called a {\it minimal lamination}. Such an Aubry-Mather set contains infinitely many ``gaps''. \\
\indent It is important to distinguish foliations from laminations. For instance, in the context of Hamiltonian twist maps and for $\omega\in\R\backslash \Q$, the $T$-invariant Lipschitz graph
$$\Gamma^{\omega} = \{ (x_i \!\!\!\!\! \mod \Z, -\p_1S(x_{i}, x_{i+1}) \ | \ x\in \mathcal{M}^{\omega}\} \subset\R/\Z \times \R$$
is an invariant quasi-periodic circle when $\mathcal{M}^{\omega}$ is topologically connected, while it is a so-called ``cantorus'' or ``remnant circle'' when $\mathcal{M}^{\omega}$ is a Cantor set. In turn, the existence of invariant circles is decisive for the occurrence of Arnol'd diffusion in the dynamics of the twist map. \\
\indent In the context of a crystal model, a minimal foliation corresponds to a continuous family of ground states of ``rotation number'' $\omega$ that can be deformed into each other by ``sliding'' the particles. Here, the rotation number has the interpretation of the average lattice spacing of the particles. A Cantor Aubry-Mather set corresponds to the existence of ``forbidden regions'' for the ground states of the crystal.

\subsection{The results in this paper}
 \noindent Periodic minimizers may come in continuous families, but typically they are isolated. It is therefore natural to ask, for a given $\omega\in\R \backslash \Q$, whether (\ref{recrel}) typically possesses a minimal foliation or rather a minimal lamination of rotation number $\omega$. The main result of this paper is that if $\omega$ is easy to approximate by rational numbers, then
 $\mathcal{M}^{\omega}$ is likely to be a minimal lamination.\\
\indent In order to quantify what it means that an irrational number is easy to approximate by rational numbers, we found it convenient to define, for constants $\gamma>0$ and $\sigma > 2$, the sets
$$\mathcal{L}_{\gamma, \sigma} := \left\{ \omega\in \R\backslash \Q\ \left| \ \exists\ \mbox{sequence} \ (p_j, q_j)\in \N\times \Z\ \mbox{with} \left| \omega - \frac{q_j}{p_j} \right| < \frac{\gamma}{p_j^{\sigma}} \right. \ \mbox{and}\ \lim_{j\to \infty} p_j=\infty \right\} .$$
Every $\mathcal{L}_{\gamma, \sigma}$ has zero Lebesque measure, but is at the same time uncountable: each $\mathcal{L}_{\gamma, \sigma}$ contains the set of Liouville numbers. \\
 \indent The first main result of this paper is that when $\omega$ is easy to approximate by rational numbers and $\mathcal{M}^{\omega}$ accidentally happens to be a minimal foliation, then this foliation can be destroyed into a lamination by an arbitrarily small smooth perturbation of the local potentials. Additionally, there are no well-ordered minimizers outside this lamination:
 \begin{theorem}\label{maintheorem2}
Let $k\in \N_{\geq 2}$ be a differentiability degree, $\gamma>0$ and $\sigma>1+2k(k+1)$ real numbers  and $\omega\in \mathcal{L}_{\gamma, \sigma}$ a rotation number. \\
\indent Then there exists a $C^k$-dense collection of local potentials that satisfy conditions {\bf A}-{\bf E} of Section \ref{problemsetup} and for which the Aubry-Mather set $\mathcal{M}^{\omega}$ is a Cantor set. These local potentials moreover do not admit minimizers in the gaps of $\mathcal{M}^{\omega}$. 
\end{theorem}
Theorem \ref{maintheorem2} is a rather direct consequence of Theorem \ref{maintheorem} below. This theorem is the second main result of this paper. Its proof is much more technical.
\begin{theorem}\label{maintheorem}
Assume that the $S_j$ are local potentials that satisfy conditions {\bf A}-{\bf E} of Section \ref{problemsetup} and let $\varepsilon>0$ be a perturbation parameter, $k\in \N_{\geq 2}$ a differentiability degree, $\gamma>0$ and $\sigma>1+2k(k+1)$ real numbers and $\omega\in \mathcal{L}_{\gamma, \sigma}$ a rotation number.
\\ \indent Then there exist local potentials $S^{\varepsilon}_j$ that satisfy conditions {\bf A}-{\bf E} and the estimate 
$$||S_j^{\varepsilon} - S_j||_{C^k} := \max_{0\leq n\leq k} \ \max_{i_1, \ldots, i_n\in \Z} \ \sup_{x\in \R^{\Z}} \left| \p_{i_1, \ldots, i_n}S_j(x)\right| \leq \varepsilon,$$
as well as a $\delta >0$ and a nonempty interval $(\eta_-, \eta_+)\subset \R$, for which the following is true. \\
\indent When $\Omega\in \R$ is a rotation number satisfying $\left|\omega- \Omega\right|\leq\delta$ and $x\in\R^{\Z}$ is a ``maximally periodic'' global minimizer of rotation number $\Omega$ of the perturbed potentials $S_j^{\varepsilon}$ and has the property that the collection
$$\{\tau_{k,l}x \ | \ k,l\in \Z\}$$ is totally ordered, then 
$$x_0\notin (\eta_-, \eta_+)\ .$$
\end{theorem}
Theorem \ref{maintheorem} implies in particular that the perturbed potentials $S_j^{\varepsilon}$ do not admit minimal foliations of rotation numbers $\Omega$ close to or equal to $\omega$. The irrational Aubry-Mather sets $\mathcal{M}^{\Omega, \varepsilon}$ of these perturbed potentials are therefore Cantor sets. \\ 
\indent
The concept of ``maximal periodicity'' will be defined in Section \ref{moreclassicalstuff}. For now, it suffices to say that periodic minimizers, as well as sequences of irrational rotation number, always have this property.

\subsection{Discussion}
In the context of Hamiltonian twist maps, destruction results for invariant circles of the kind in Theorems \ref{maintheorem2} and \ref{maintheorem}, as well as theorems that guarantee the total absence of invariant circles, go under the name ``converse KAM theorems'', cf.  \cite{percival}, \cite{MacKayMeiss} and \cite{MackayMeissStark}. This is because the well-known KAM theory provides persistence results for invariant circles of very irrational, e.g. ``Diophantine'', rotation numbers, see \cite{arnoldmathmethods}.
\\ \indent
One could argue that in the context of general recurrence relations of the form (\ref{recrel}) this terminology is not justified: we are not aware of any KAM-type persistence results for minimal foliations of fully general recurrence relations of the form (\ref{recrel}). Exceptions are   persistence theorems for minimal foliations of certain elliptic PDEs in \cite{moser88} and \cite{moserbrasil} and for recurrence relations that are close to a Hamiltonian twist map in \cite{llavecalleja}. Theorem \ref{maintheorem2}  only provides conditions under which a KAM theorem for (\ref{recrel}) can certainly not be true.
\\ \indent
The destruction results we know of are also only valid for Hamiltonian twist maps. One of these results for twist maps is very similar to Theorem \ref{maintheorem2} and was obtained by Mather \cite{Matherdestruction}. Like our proof of Theorem \ref{maintheorem2}, the proof in \cite{Matherdestruction} is variational. It relies on an earlier result obtained in \cite{Matherpeierls} that establishes a modulus of continuity for the so-called Peierls barrier function. We do not make use of or even define the Peierls barrier function in this paper, but we imagine that with our techniques one could also prove its continuity. \\
\indent We do in fact follow some of the ideas in \cite{Matherdestruction}, but there is an important point at which we deviate from Mather's ideas. This is necessary, because there is a crucial difference between recurrence relations that stem from a Hamiltonian twist map and general recurrence relations of the form (\ref{recrel}). \\
\indent In fact, when (\ref{recrel}) describes the orbits of a Hamiltonian twist map, then one can prove that two different minimizers can cross at most once, see \cite{AubryDaeron}. The proof of continuity of the Peierls barrier function in \cite{Matherpeierls} heavily depends on this single-crossing property, but for general recurrence relations of the form (\ref{recrel}) this property is not true. To illustrate this, one can consider the solutions of (\ref{FKnextstat}) with $V\equiv 0$, i.e. the solutions of 
$$x_{i+2} + x_{i+1}  - 4x_i + x_{i-1} +x_{i-2} =0 \ \mbox{for all}  \ i \in \Z \ .$$
The general solution to this relation can easily be found by an Ansatz $x_i=\alpha^{i}$ and it reads
\begin{align}\label{FKlinear}
x_i = c_0 + c_1\cdot i + c_2\left( -\frac{3}{2}+\frac{1}{2}\sqrt{5} \right)^i + c_3\left( -\frac{3}{2}-\frac{1}{2}\sqrt{5} \right)^i \ \mbox{with} \ c_0, c_1, c_2, c_3 \in \R \ .
\end{align}
There obviously are pairs of solutions that cross infinitely often. At the same time, all solutions given in (\ref{FKlinear}) are global minimizers. This is true because one can check that for all $x\in\R^{\Z}$ and for all finite subsets $B\subset \Z$, the map $y\mapsto W(x+y)-W(x)$ on the space of sequences supported on $B$, is convex. \\ 
\indent In our proof, the single-crossing property for minimizers is therefore replaced by a property that holds more generally. This property will be formulated and proved in Theorem \ref{regularitytheorem} and it consists of a ``near-periodicity'' result for maximally periodic Birkhoff sequences. Theorem \ref{regularitytheorem} makes that our proof and that in \cite{Matherdestruction} have a quite different character. We would say that our proof is a bit simpler. We are moreover confident that our proof allows for a generalization to variational problems on lattices, such as those discussed in \cite{llave-valdinoci}, \cite{llave-lattices} or \cite{MramorRink1}. \\
\indent
Although we prove in this paper that the set of local potentials without a minimal foliation of certain rotation numbers is dense, we do not show that this set is open, residual or in any other sense generic in the $C^k$-topology. This remains an interesting open question.
\\
\indent Moreover, it is not clear to us at this point whether Theorem \ref{maintheorem2} is optimal. That is, whether a minimal foliation of rotation number $\omega$ will always persist under arbitrarily small $C^k$-perturbations as soon as $\omega\notin \bigcup_{\gamma>0} \mathcal{L}_{\gamma, \sigma}$ and $\sigma \leq 1+2k(k+1)$. This would be the content of a KAM theorem. \\
\indent In this regard, it is also interesting to recall the famous result of Brjuno \cite{brjuno} and Yoccoz \cite{yoccoz} concerning Siegel's problem. This result says that every holomorphic map of the form
$$z\mapsto e^{2\pi i \alpha} z + \mbox{nonlinearity}$$
on an open neighborhood of $0$ in $\C$ can be linearized locally near $0$, if and only if $\alpha$ is a  Brjuno number, that is unless $\alpha$ admits extremely good rational approximations. Analogously, one may expect that every minimal foliation for (\ref{recrel}) can be destroyed by an arbitrarily small holomorphic perturbation of the local potentials if and only if the rotation number of that foliation is not a Brjuno number. It would be interesting to investigate if this is true. We refer to \cite{Fornianalytic} for a partial result in this direction for holomorphic twist maps.

\subsection{Outline of this paper}
In Section \ref{moreclassicalstuff} we outline some more classical Aubry-Mather theory that is necessary for the understanding of this paper. Section \ref{moreperiodicsection} contains our near-periodicity Theorem \ref{regularitytheorem}. Then, in Section \ref{periodicdestructionsection}, we discuss how foliations of periodic minimizers can be destroyed. Although the results in this section are rather obvious, they include some quantitative estimates that are important for later. Sections \ref{firststepssection}, \ref{periodicpersistencesection} and \ref{proofsection} are dedicated to the proof of Theorem \ref{maintheorem}, where in Section \ref{firststepssection} we still follow \cite{Matherdestruction} quite closely. Finally, in Section \ref{destructionproofsection}, we prove Theorem \ref{maintheorem2}.

\subsection{Acknowledgement} We would like to acknowledge Henk Broer for several useful comments that helped us formulate our results more clearly and we want to thank our colleagues of the Department of Mathematics at VU University Amsterdam for their continuous support. This research was partially funded by the Dutch Science Foundation NWO.

\section{Classical Aubry-Mather theory}\label{moreclassicalstuff}
In this Section, we summarize some standard concepts and constructions from classical Aubry-Mather theory that were not discussed in the introduction, but that are essential for the understanding of this paper. For the proofs of our statements we refer to \cite{MramorRink1}. In one form or another, much of this section can also be found in the standard references \cite{AubryDaeron}, \cite{gole01}, \cite{MatherTopology} and \cite{MatherForni}.

\subsection{The Birkhoff property}
One can define a partial ordering on the space of sequences as follows:
\begin{definition} 
For $x, y \in \R^{\Z}$ we write
\begin{itemize}
\item $x \leq y$ if $x_i \leq y_i$ for every $i\in \Z$. 
\item $x<y$ if $x\leq y$, but $x \neq y$. We then say that $x$ and $y$ are {\it weakly ordered}.
\item $x \ll y$ if $x_i < y_i$ for every $i\in \Z$. We say that $x$ and $y$ are {\it strictly ordered}.
\end{itemize}
Similarly for $\geq$, $>$ and $\gg$. 
\end{definition}
\noindent Recall the definition of the shift operators $\tau_{k,l}:\mathbb{R}^{\mathbb{Z}}\to \mathbb{R}^{\mathbb{Z}}$ given in (\ref{definitiontau}). The partial orderings defined above, allow us to make the following definition.
\begin{definition}\label{defbirkhoff}
A sequence $x \in \R^{\Z}$ is called a {\it Birkhoff sequence} or a {\it well-ordered sequence}, if the collection 
$$\{ \tau_{k,l}x\ | \ k,l\in \Z\}$$
is totally ordered. In other words, if
for all $k, l \in \Z$, either $\tau_{k,l}x \geq x$ or $\tau_{k,l}x \leq x$. \\
\indent The collection of Birkhoff sequences will be denoted 
$$\mathcal{B}\subset \R^{\Z}$$
and it inherits the topology of pointwise convergence. 
\end{definition}
\noindent Definition \ref{defbirkhoff} says that the graph of a Birkhoff sequence $x$ does not cross any of its integer translates.
\begin{example}
When $h:\R/\Z\to\R/\Z$ is an orientation preserving circle homeomorphism, then it admits a {\it lift} to a strictly increasing continuous map $H:\R\to\R$ that satisfies $H(\xi+1)=H(\xi)+1$ and $H(\xi)\!\!\! \mod \! 1 = h(\xi \!\!\! \mod \! 1)$. \\
\indent If we now denote by $x(\xi):\Z\to\R$ the $H$-orbit of $\xi\in \R$, defined by $x(\xi)_i:=H^i(\xi)$, then $x(\xi_1)\gg x(\xi_2)$ if and only if $\xi_1>\xi_2$. In other words, the collection $\{x(\xi)\ | \ \xi\in \R\}\subset \R^{\Z}$ of orbits of $H$, is strictly ordered. In particular, every $x(\xi)$ is a Birkhoff sequence.
\end{example}
A famous theorem of Poincar\'e says that every circle homeomorphism has a rotation number. In fact, this result only depends on the Birkhoff property of the orbits. Hence, Poincar\'e's theorem can also be put in the following general form:
\begin{proposition}\label{omegaproposition}
Let $\Gamma\subset \R^{\Z}$ be a totally ordered and shift-invariant collection of sequences. Then every $x\in \Gamma$ has a rotation number, say $\omega$, and this rotation number is the same for every element of $\Gamma$. More precisely, it holds for every $x\in \Gamma$ that
\begin{align}\label{poincareestimate}
|x_i - x_0-\omega\cdot i| \leq1\ \mbox{for all} \ i\in \Z\ .
\end{align}
\end{proposition}
Proposition \ref{omegaproposition} implies that we can decompose 
$\mathcal{B}=\bigcup_{\omega\in\R}\mathcal{B}_{\omega}$, where
$$\mathcal{B}_{\omega}:=\{x\in \mathcal{B} \ | \ \mbox{the rotation number of}\ x \ \mbox{equals}\ \omega\}\ . $$ 
Proposition \ref{omegaproposition} has two more or less direct consequences that we state here without proof.
\begin{proposition}\label{limitrotationnumbers}
When $x^n\in \mathcal{B}_{\omega_n}$ is a sequence of Birkhoff sequences so that $\lim_{n\to\infty} x^n=x^{\infty}$ pointwise, then $x^{\infty}\in \mathcal{B}$, the limit $\lim_{n\to\infty}\omega_n=\omega$ exists and the rotation number of $x^{\infty}$ equals $\omega$. In other words, $\mathcal{B}$ is closed and the map $x \mapsto \omega, \ \mathcal{B} \to \mathbb{R}$ is continuous. 
\end{proposition}

\begin{proposition}\label{compactness}
Let $K \subset \R$ be compact and let $\mathcal{B}_K := \bigcup_{\omega \in K} \mathcal{B}_{\omega}$. Furthermore, let us identify every sequence $x$ with its vertical translates $\tau_{0,\Z}x=x+\Z$. Then $\mathcal{B}_K /\Z$ is compact in the topology of pointwise convergence.
\end{proposition}
 
\noindent We conclude this section with two ``number-theoretic'' results that we will need later. The first one expresses that the rotation number $\omega$ of a Birkhoff sequence $x$ determines almost completely how the collection 
$\{ \tau_{k,l}x\ |\ k,l\in \Z\}$ is ordered. 
{\begin{proposition}\label{numbertheory}
Let $\omega\in \R$ and $x \in \mathcal{B}_\omega$. If $-\omega\cdot k+l>0$, then $\tau_{k,l}x > x$ and if $-\omega \cdot k+l<0$, then $\tau_{k,l}x < x$.
\end{proposition}
When $-\omega\cdot k+l=0$ for some nonzero $k,l\in \Z$, then Proposition \ref{numbertheory} does not say how $\tau_{k,l}x$ and $x$ are ordered. This situation can occur when $\omega\in\Q$. To exclude this ambiguity, we make the following definition:
\begin{definition}\label{maxperdef}
A Birkhoff sequence $x\in\mathcal{B}_{\omega}$ is called {\it maximally periodic} if for all $k,l\in \Z$ with $-\omega\cdot k + l = 0$ it holds that $\tau_{k,l}x=x$.
\end{definition}
It is clear that when $\omega \in\R \backslash \Q$, then every $x\in \mathcal{B}_{\omega}$ is automatically maximally periodic. \\
\indent
For the second ``number-theoretic'' result, we recall the definition of the space of $\X_{p,q}$ of periodic sequences given in (\ref{Xpqdef}). We will denote the collection of periodic Birkhoff sequences of periods $(p,q)$ by 
 $$\mathcal{B}_{p,q}:=\mathcal{B}\cap \X_{p,q} \ .$$
Because the elements of $\X_{p,q}$ have rotation number $\omega=\frac{q}{p}$, we have that $\mathcal{B}_{p,q} \subset \mathcal{B}_{q/p}$.  
The final result of this section is therefore a rather straightforward application of Proposition \ref{numbertheory} to the case that $\omega=\frac{q}{p}$.
 \begin{proposition}\label{maxrelper} Periodic Birkhoff sequences are as periodic as they can be. More precisely, when $n\in \N$ and $(p,q)\in \N\times \Z$, then $\mathcal{B}_{np,nq}=\mathcal{B}_{p,q}$. In other words, the periods of a periodic Birkhoff configuration can be chosen relative prime.  
\end{proposition}
As a consequence of Proposition \ref{maxrelper}, an $x\in \mathcal{B}_{q/p}$ is maximally periodic if and only if it is periodic. We also remark that there do exist Birkhoff sequences of rational rotation number $\frac{q}{p}$ that are not periodic.
 
\subsection{More about minimizers}\label{classicalresults}
Birkhoff sequences are important in the study of recurrence relations of the form (\ref{recrel}), because many of the global minimizers of (\ref{recrel}) have the Birkhoff property. For instance, all periodic minimizers do. We will explain this below.
\\ \indent 
The first thing to remark is that a $(p,q)$-periodic sequence $x\in \X_{p,q}$ is a solution to (\ref{recrel}) if and only if it is a stationary point of the periodic action function
$$W_{p}:\X_{p,q}\to\R\ \mbox{defined by}\ W_p(x):=\sum_{j=1}^p S_j(x)\ .$$
Because $\X_{p,q}$ is finite-dimensional and $W_p(x)$ is a finite sum, these stationary points are well-defined and in particular one calls an $x\in \X_{p,q}$ a {\it periodic minimizer} or {\it $(p,q)$-minimizer} if it minimizes $W_p$ over $\X_{p,q}$. \\
\indent The following proposition summarizes all we need to know about periodic minimizers. For a full proof of this proposition, we refer to \cite{MramorRink1}.
\begin{proposition}
\label{periodicminimizersproposition}
For all $(p,q)\in \N\times \Z$, the collection of $(p,q)$-minimizers is nonempty, closed under pointwise convergence, shift-invariant and strictly ordered. In particular, every $(p,q)$-minimizer has the Birkhoff property. \\
\indent Moreover, $x\in \X_{p,q}$ is a $(p,q)$-minimizer if and only if it is an $(np,nq)$-minimizer for any $n\in \N$, if and only if it is a global minimizer.
\end{proposition}
\begin{proof}[{\bf Sketch}] The invariance of the $S_j$ under $\tau_{0,1}$ implies that the function $W_p$ descends to a function on $\X_{p,q}/\Z$. Condition {\bf C} implies that this function is coercive. This guarantees that a $(p,q)$-minimizer exists. \\
\indent The set of $(p,q)$-minimizers is closed because condition {\bf A} implies that $W_p:\X_{p,q}\to\R$ is continuous.
\\
\indent Condition {\bf B} implies that $W_p(\tau_{k,l}x)=W_p(x)$ for all $k,l\in \Z$ and all $x\in \X_{p,q}$. Thus, the collection of $(p,q)$-minimizers is shift-invariant.
\\ \indent 
 The strict ordering deserves some more explanation. This property is sometimes called ``Aubry's lemma''. It follows from two observations that one derives from condition {\bf D}. The first observation is a ``weak maximum principle''. To formulate it, one defines for arbitrary $x, y\in \X_{p,q}$, the sequences $x\wedge y$ and $x \vee y$ in $\X_{p,q}$ by 
$$(x\wedge y)_i:=\min\{x_i, y_i\}\ \mbox{and} \ (x \vee y)_i:=\max\{x_i, y_i\}\ .$$
With the help of the first part of condition {\bf D}, the weak monotonicity condition, one can then compute that
\begin{align}
W_p(x \wedge y)+W_p(x \vee y) \leq W_p(x) + W_p(y) \ .
\end{align}
In particular, when $x$ and $y$ are $(p,q)$-minimizers, then so are $x\wedge y$ and $x \vee y$. This is the weak maximum principle. 
\\ \indent Closely inspecting (\ref{recrel}) and using condition {\bf D} in its strong form, one can moreover show that two solutions of (\ref{recrel}) can not ``touch''. More precisely, when $x<y$ are two nonidentical, not necessarily periodic, weakly ordered solutions to (\ref{recrel}), then actually $x\ll y$. In other words, two weakly ordered solutions to (\ref{recrel}) must automatically be strictly ordered. This is the ``strong maximum principle''.
 \\
\indent Now one argues as follows. Suppose that $x, y\in \X_{p,q}$ are two nonidentical minimizers that are not strictly ordered, i.e. that $x$ and $y$ ``cross'' or ``touch''. By the weak maximum principle, this implies that then $x \wedge y$ and $x$ form a pair of weakly ordered but not strictly ordered minimizers. But by the strong maximum principle this is impossible. We conclude that the set of $(p,q)$-minimizers is strictly ordered.
\\
\indent The final statement of Proposition \ref{periodicminimizersproposition} is related to Proposition \ref{maxrelper}. We omit the proof. It is not completely trivial. 
\end{proof}
Global minimizers of irrational rotation numbers can now be constructed as limits of periodic minimizers. This works thanks to the following proposition.
\begin{proposition}\label{limitminimizers}
When $x^n\in \R^{\Z}$ is a sequence of global minimizers and $\lim_{n\to \infty} x^n=x^{\infty}$ pointwise, then also $x^{\infty}$ is a global minimizer.
\end{proposition}
Let us now sketch the well-known procedure for constructing minimizers of arbitrary rotation numbers. 
\\ \indent 
Given $\omega\in\R \backslash \Q$ let us choose a sequence $\frac{q_n}{p_n}$ of rational numbers so that $\lim_{n\to\infty}\frac{q_n}{p_n}=\omega$. Let $x^{p_n, q_n}\in \X_{p_n, q_n}$ be a corresponding sequence of periodic minimizers. We have seen that these exist and have rotation number $\frac{q_n}{p_n}$. Moreover, each of them is Birkhoff, i.e. $x^{p_n, q_n} \in \mathcal{B}_{p_n, q_n}$. By shift-invariance, one may assume that $x^{p_n, q_n}_0\in [0,1]$ and hence by Proposition \ref{compactness}, there then is a subsequence $x^{p_{n_j}, q_{n_j}}$ that limits pointwise to a sequence $x^{\infty}\in \mathcal{B}$. By Proposition \ref{limitminimizers}, this $x^{\infty}$ is a global minimizer, while by Proposition \ref{limitrotationnumbers} it has rotation number $\omega$. We have proved:
\begin{theorem}
For every $\omega\in \R$ there exists a Birkhoff global minimizer of rotation number $\omega$. If $\omega\in \Q$, then this global minimizer can be chosen periodic.
\end{theorem}
Especially when $\omega\in\R \backslash \Q$, the existence of one minimizer $x\in \mathcal{B}_{\omega}$ enforces the existence of many more, namely also all the $\tau_{k,l}x$ are minimizers. Because $\omega\notin \Q$, Proposition \ref{numbertheory} guarantees that $\tau_{k,l}x\neq \tau_{K,L}x$ unless $k=K$ and $l=L$. In view of Proposition \ref{limitminimizers}, the set of translates of $x$ and their pointwise limits
$$\mathcal{M}(x)=\overline{ \{\tau_{k,l}x\ | \ (k,l)\in \Z\times \Z \} }$$
therefore forms a very large set of minimizers. The collection $\mathcal{M}(x)$ is shift-invariant, closed under pointwise convergence and, due to the strong maximum principle, strictly ordered. In particular, all the elements of $\mathcal{M}(x)$ have rotation number $\omega$. 
\\
\indent The {\it Aubry-Mather set} $\mathcal{M}^{\omega}$ of rotation number $\omega$ is now defined as the minimal subset of $\mathcal{M}(x)$. Minimality here means that $\mathcal{M}^{\omega}$ is nonempty and does not contain any proper nonempty subset that is also shift-invariant and closed under pointwise convergence. It was shown by Bangert \cite{bangert87} that $\mathcal{M}^{\omega}$ actually does not depend on the choice of $x$. In the special case of twist maps, this latter fact was already known to Aubry and Le Daeron \cite{AubryDaeron}. \\
\indent We summarize the properties of the Aubry-Mather set in the following theorem:
\begin{theorem}
For $\omega\in\R \backslash \Q$, the Aubry-Mather set $\mathcal{M}^{\omega}$ is the unique nonempty, closed under pointwise convergence, shift-invariant, strictly ordered and minimal collection of minimizers of rotation number $\omega$. \\
\indent Every element of $\mathcal{M}^{\omega}$ is the pointwise limit of periodic minimizers.  Moreover, $\mathcal{M}^{\omega}$ is either topologically connected or a Cantor set.
\end{theorem}
We remind the reader that a Cantor set is a topological space that is closed, perfect and totally disconnected. More precisely, a topological space $Y$ is called {\it perfect} if every $y\in Y$ is the limit of points in its complement $Y\backslash \{y\}$, whereas $Y$ is called {\it totally disconnected} if for every $y_1, y_2 \in Y$ one can decompose $Y$ as the disjoint union $Y=Y_1\cup Y_2$ of closed subsets $Y_1$ and $Y_2$ with $y_1\in Y_1$ and $y_2\in Y_2$.\\
\indent
When $\mathcal{M}^{\omega}$ is connected, one says that it forms a {\it minimal foliation}. In case $\mathcal{M}^{\omega}$ is a Cantor set, we say that it forms a {\it minimal lamination}. This is because $\mathcal{M}^{\omega}$ then has many {\it gaps}. More precisely, one can then show that for every $\xi_1< \xi_2$ in $\R$ there are $\xi_1\leq \eta_1<\eta_2\leq \xi_2$ such that there does not exist any $y\in \mathcal{M}^{\omega}$ with $\eta_1 <y_0<\eta_2$.

\section{More about periodicity of Birkhoff sequences}\label{moreperiodicsection}
In this section we study the periodicity properties of Birkhoff sequences in detail. For the purpose of this paper, the main result of this section is Theorem \ref{regularitytheorem}. It is a quantitative near-periodicity result for maximally periodic Birkhoff sequences.  To the best of our knowledge this theorem is new. It will be a key ingredient for the proof of Theorem \ref{maintheorem}. As was explained in the introduction, it replaces the ``single-crossing'' property for minimizers that is used in \cite{Matherpeierls} and \cite{Matherdestruction}.
 
\subsection{Group theoretic remarks}\label{birkhoffpersection}
We first make some group theoretic remarks. It appears to us that most of these remarks have been made before in one form or another, see for instance \cite{Matherpeierls}. \\
\indent We would like to remind the reader that one can think of the shift operators $\tau_{k,l}$ as defining a group action of $\Z\times \Z$ on the space of sequences:
$$\tau: (\Z\times \Z)\times \R^{\Z}\to \R^{\Z}, \ ((k,l),x)\mapsto \tau_{k,l}x\ . $$
With this interpretation, $\X_{p,q}$ consists precisely of the sequences that are fixed by the subgroup 
$$J_{p,q}: = \left\{(np, nq)\ | \ n\in \Z\right\} \subset \Z\times \Z\ .$$
Because $\Z\times \Z$ is Abelian, when $\tau_{p, q}x=x$, then also $\tau_{p,q}(\tau_{k,l}x)=\tau_{k,l}(\tau_{p,q}x)=\tau_{k, l}x$, and thus $\tau$ leaves $\X_{p,q}$ invariant.  Moreover, because the elements of $J_{p,q}$ fix all elements of $\X_{p,q}$, we have that when $x\in \X_{p,q}$ and $(k,l)=(K,L)+ (np, nq)$ for some integer $n$, then  $\tau_{k,l}x=\tau_{K,L}(\tau_{p,q}^nx)=\tau_{K,L}x$. \\
\indent Together, these observations show that $\tau$ gives rise to an action of $(\Z\times \Z)/J_{p,q}$ on $\X_{p,q}$.  We now have the following
\begin{lemma}\label{tauaction}
When $(p,q)\in \N\times \Z$, $(k,l)\in \Z\times \Z$ and $x\in\X_{p,q}$, then
\begin{align}\label{l1periodic}
||\tau_{k,l}x-x||_{l_1(p)}: = \sum_{i=1}^{p}|(\tau_{k,l}x-x)_i| \geq \left| pl-qk \right|, 
\end{align}
with equality holding in (\ref{l1periodic}) when $x\in \mathcal{B}_{p,q}$.\\
\indent Thus, the action of $(\mathbb{Z}\times \Z)/ J_{p,q}$ on $\X_{p,q}$ is free if and only if $p$ and $q$ are relative prime. 
\end{lemma} 
 
\begin{proof} 
Let $x\in \X_{p,q}$ and let $(k, l) \in \Z\times \Z$ be given. Then we can remark that
$$\tau_{k,l}^px = \tau_{pk,pl}x= \tau_{0,pl-qk} \circ \tau_{pk,qk} x =  \tau_{0,pl-qk} \circ \tau_{p,q}^k x =  \tau_{0, pl-qk} x =  x + (pl-qk)\ .$$
Using that $p\geq 1$, this shows that 
$$||\tau^p_{k,l}x-x||_{l_1(p)}= p \cdot \left|pl-qk\right|\ .$$
The next remark is that $\tau_{k,l}^{j+1}x-\tau^j_{k,l}x=\tau_{k,0}(\tau_{k,l}^jx-\tau_{k,l}^{j-1}x)$ and thus, by induction, that $||\tau_{k,l}^{j+1}x-\tau_{k,l}^jx||_{l_1(p)}=||\tau_{k,l}x-x||_{l_1(p)}$. For a general $x\in \X_{p,q}$, this implies that 
$$||\tau_{k,l}^px-x||_{l_1(p)}\leq ||\tau^{p}_{k,l}x-\tau_{k,l}^{p-1}x||_{l_1(p)} + \ldots + ||\tau_{k,l}x-x||_{l_1(p)}=p\cdot ||\tau_{k,l}x-x||_{l_1(p)}\ .$$ 
Hence, it holds that $||\tau_{k,l}x-x||_{l_1(p)}\geq |pl-qk|$.\\
\indent When $x\in \mathcal{B}_{p,q}$, then either $\tau_{k,l}x>x$ or $\tau_{k,l}x=x$ or $\tau_{k,l}x<x$. In the first case, $\tau_{k,l}^jx=\tau_{k,l}^{j-1}(\tau_{k,l}x)>\tau_{k,l}^{j-1}x$ for all $j$. Similarly, in the second case, $\tau_{k,l}^{j}x<\tau_{k,l}^{j-1}x$ for all $j$, while in the third case, $\tau_{k,l}^jx=\tau_{k,l}^{j-1}x$ for all $j$.\\
\indent In either of the three cases above it follows that $||\tau_{k,l}^px-x||_{l_1(p)} = \sum_{j=1}^{p}||\tau^{j}_{k,l}x-\tau^{j-1}_{k,l}x||_{l_1(p)}$. This implies that when $x\in \mathcal{B}_{p,q}$, then $||\tau_{k,l}x-x||_{l_1(p)}=|pl-qk|$. \\
\indent When $p$ and $q$ are relative prime, then $(k,l)$ represents a nontrivial equivalence class in $(\Z\times \Z)/J_{p,q}$ if and only if $pl-qk\neq 0$. In this case, $||\tau_{k,l}x-x||_{l_1(p)}\geq 1$ for all $x\in \X_{p,q}$ and thus the action is free. In fact, we explicitly observe here that the action is properly discontinuous. \\
\indent Conversely, if $p$ and $q$ are not relative prime, then there exist $k$ and $l$ with $pl-qk=0$ but $(k,l)\neq (np, nq)$ for any $n\in \Z$. The latter means that $(k,l)$ represents a nontrivial element of $(\Z\times \Z)/J_{p,q}$, whereas for any $x\in \X_{p,q}$ that is Birkhoff, it holds that $||\tau_{k,l}x-x||_{l_1(p)}=0$. This means that the action is not free.
\end{proof} 
\noindent 
We recall that when $p$ and $q$ are relative prime, then there exist $s,t\in \Z$ for which
\begin{align}\label{pqst}
pt-qs=1\ . 
\end{align}
Modulo transformations of the form $(s, t)\mapsto (s,t)+(np, nq)$ these integers are unique. Using (\ref{pqst}), it is easy to verify that for all $(k,l)\in \Z\times \Z$ one has 
$$(k,l)= (kt-ls)(p,q) + (pl-qk)(s,t)\ .$$
This shows that $(k,l) \equiv (pl-qk)(s,t)$ modulo $J_{p,q}$. In particular, $(\Z\times \Z)/ J_{p,q}$ is generated by the equivalence class of $(s,t)$. \\
\indent For $p, q, s$ and $t$ satisfying (\ref{pqst}), we will denote by 
\begin{align}\label{upqdefinition}
U_{p,q}:=\tau_{s,t}: \X_{p,q}\to\X_{p,q}
\end{align}
the corresponding translation map. It is characterized by the fact that 
\begin{align}\label{propertiesU}
\mbox{for} \ x\in \mathcal{B}_{p,q}\ \mbox{it holds that}\ U_{p,q}x> x \  \mbox{and} \ ||U_{p,q}x-x||_{l_1(p)}=1\ .
\end{align}
The first statement in (\ref{propertiesU}) holds because $j\mapsto U_{p,q}^jx \in \mathcal{B}_{p,q}$ is monotone when $x\in \mathcal{B}_{p,q}$, while $U_{p,q}^px=x+(pt-qs) = x+1 \gg x$.\\ 
\indent The second statement in (\ref{propertiesU}) directly follows from (\ref{l1periodic}) and (\ref{pqst}).

\subsection{A near-periodicity theorem}\label{nearperiodicitysection}
The aim of this section is to prove Theorem \ref{regularitytheorem} below. It can be interpreted as a quantitative near-periodicity result for Birkhoff sequences. \\
\indent To motivate this theorem, let us consider, for some $\omega\in \R$, the linear sequence $x^{\omega}$ defined by $x^{\omega}_j:=x_0+ \omega\cdot j$. For integers $(p, q)\in \N\times \Z$, it then holds for all $j\in \Z$ that 
\begin{align}\label{easycomputation}
(\tau_{p, q}x^{\omega})_j - x^{\omega}_j = x^{\omega}_{j-p} -x^{\omega}_j + q = q -\omega p \ .
\end{align}
Hence, when $q-\omega p$ is small, then  $x^{\omega}$ is ``almost'' $(p, q)$-periodic. Theorem \ref{regularitytheorem} says that such a property is true for all $x\in \mathcal{B}_{\omega}$ that are ``maximally periodic'', see Definition \ref{maxperdef}.\\
\indent The precise statement is the following:
\begin{theorem}[A near-periodicity theorem]\label{regularitytheorem}
Let $p \in \N$ and $q \in \Z$ be relative prime, $\omega\in \R$ and $r\geq 1$ and let $i_1\leq i_2$ be integers. \\
\indent We denote by $\lceil \alpha \rceil$ the smallest integer bigger than or equal to $\alpha$ and we define 
$$a=a(p, q,\omega,i_2-i_1):=\left\lceil (i_2-i_1) \left|q- \omega p \right|\right\rceil\ .$$
Assume that $x\in\mathcal{B}_{\omega}$ is maximally periodic. Then there exists an $i_0\in \Z$ so that for all integers $m, n$ with $i_1+r\leq i_0+m p, i_0 + n p\leq i_2-r+1$, we have
\begin{align}\label{pigeintheorem}
|| \tau^{-m}_{p,q}x-\tau_{p,q}^{-n}x||_{l_1[i_0-r, i_0+r-1]}:=\sum_{j=i_0-r}^{i_0+r-1} |x_{m p+j} - x_{n p+j} - (m-n) q| \leq \frac{2r\cdot a}{p}\ .
\end{align}

\end{theorem}
When $x=x^{\omega}$ is a linear sequence of rotation number $\omega$, then estimate (\ref{pigeintheorem}) holds for every $i_0\in \Z$. This follows from a computation similar to (\ref{easycomputation}). Thus, Theorem \ref{regularitytheorem} can be seen as a generalization of (\ref{easycomputation}).
\\
\indent It should also be noted that the integer $i_0$ in Theorem \ref{regularitytheorem} is not unique. We will later always choose $-p<i_0\leq 0$.\\
\indent 
To prove Theorem \ref{regularitytheorem}, we will need two preliminary results. The first one is a direct application of the observations we made in Section \ref{birkhoffpersection} and the pigeonhole principle:
\begin{proposition}\label{pigeonholeproposition}
Let $(p, q)\in \N\times \Z$ be relative prime, let $y\in \mathcal{B}_{p, q}$ and let $r\geq 1$ and $a\geq 0$ be integers. We denote $U_{p,q} = \tau_{s,t} :\X_{p,q}\to \X_{p,q}$ with $pt-qs=1$. \\
\indent Then there exists an $i_0\in \Z$ for which
$$||U_{p,q}^ay-y||_{l_1[i_0-r, i_0+r-1]} = \sum_{j=i_0-r}^{i_0+r-1} |(U_{p,q}^ay)_j-y_j| \leq \frac{2r\cdot a}{p}\ .$$
\end{proposition}
\begin{proof}
According to Lemma \ref{tauaction} it holds for all $y\in \mathcal{B}_{p,q}$ that
\begin{align} \label{l1property} 
||U_{p,q}^ay-y||_{l_1(p)} = |a(pt-qs)|=a\ .
\end{align}
We claim that this implies that there exists an $ i_0$ for which 
\begin{align}\label{confinementagain}
\sum_{j=i_0-r}^{i_0+r-1}|(U_{p,q}^ay)_j - y_j| \leq \frac{2r\cdot a}{p}\ .
\end{align}
This is an easy consequence of the pigeonhole principle. Indeed, if it were true that $\sum_{j=i-r}^{i+r-1}|(U_{p,q}^ay)_j - y_j| > \frac{2r\cdot a}{p}$ for all $i\in\Z$, then it would hold that 
$$2r\cdot a < \sum_{i=1}^{p} \sum_{j=i-r}^{i+r-1} \left| (U^a_{p,q}y)_j-y_j \right| = \sum_{j=-r}^{r-1} \sum_{i=1}^{p} |(U^a_{p,q}y)_{i+j}-y_{i+j}| = 2r ||U_{p,q}^ay - y||_{l_1(p)}\ .$$ 
The first equality is an ordinary re-summation and the second equality holds because $y$ is $(p,q)$-periodic. This is a contradiction.  
\end{proof}
The second preliminary result of this section tells us how we can squeeze part of a maximally periodic Birkhoff sequence in between certain translates of a periodic Birkhoff sequence of another rotation number:
\begin{theorem}[Confinement]\label{confinement}
Let $(p,q) \in  \N\times \Z$ and $\omega\in \R$ be given and assume that $p$ and $q$ are relative prime. We again denote $U_{p,q} = \tau_{s,t} :\X_{p,q}\to \X_{p,q}$ with $pt-qs=1$. \\
\indent Then for every maximally periodic $x \in \mathcal{B}_{\omega}$ and for all integers $i_1 \leq i_2$, there exists a $y\in \mathcal{B}_{p,q}$ so that
\begin{align}\label{confinementstatement}
y_j \leq x_j \leq (U_{p,q}^ay)_j \ \mbox{for every} \ i_1\leq j \leq i_2\ .
\end{align}
Here, 
$$a=a(p,q,\omega,i_2-i_1)= \left\lceil (i_2-i_1)\left| q -\omega p \right|\right\rceil\ .$$
\end{theorem}
\begin{proof} 
We will prove the theorem in the case that $\frac{q}{p}\leq \omega$. The case that $
\frac{q}{p}\geq \omega$ is similar.\\
\indent So let $x\in \mathcal{B}_{\omega}$ be maximally periodic and let $i_1\leq i_2$. We will first prove the theorem when $x$ happens to be linear, that is when $x=x^{\omega}\in \mathcal{B}_{\omega}$, with $x^{\omega}$ defined as
$$x^{\omega}_j := x_{i_1} + \omega(j-i_1)\ .$$
In this case, also $y$ can be chosen linear, namely $y=y^{q/p}$ does the job, with
$$y^{q/p}_j:= x_{i_1} + \frac{q}{p}(j-i_1)\ .$$
Indeed, for all $j\geq i_1$ it holds that $y^{q/p}_j \leq x^{\omega}_j$, because $\frac{q}{p}\leq \omega$. \\
\indent Moreover, using that $pt-qs=1$, one computes that for every integer $a\geq 0$,
$$(U_{p,q}^ay^{q/p})_j  = x_{i_1} + \frac{q}{p}(j-as - i_1) +at =  x_{i_1} +\frac{q}{p}(j-i_1)+ \frac{a}{p}\ ,$$
so that 
$$x^{\omega}_j \leq (U_{p,q}^ay^{q/p})_j \ \mbox{so long as} \ a \geq (j-i_1) \left( \omega p - q \right) \ . $$
In particular, $x^{\omega}_j \leq (U_{p,q}^ay^{q/p})_j$ for all $i_1\leq j \leq i_2$ if we choose $a = \lceil (i_2-i_1) \left| q - \omega p \right|\rceil$. This proves the theorem in case $x\in \mathcal{B}_{\omega}$ is linear. \\
\indent Now we consider the situation that $x \in \mathcal{B}_{\omega}$ is nonlinear, but still maximally periodic. Then we define a function $\psi:\R\to\R$ that sends the linear sequence $x^{\omega}$ to the nonlinear sequence $x$. \\
\indent More precisely, we first define $\psi$ on the set 
$$\Sigma_{x^{\omega}}:=\{x_{i_1} + \omega (k-i_1) + l \ | k,l\in \Z\}\subset \R\ .$$
This is done by setting 
$$\psi(x_{i_1}+\omega(k-i_1) + l ):= x_{k}+l \ .$$
The function $\psi$ is well-defined because when $x_{i_1}+\omega(k-i_1) +l = x_{i_1}+\omega (K-i_1)+L$, then $\omega(k-K)+l-L=0$ and hence, since $x$ is maximally periodic, $\tau_{K-k,l-L}x=x$, that is $\tau_{-k,l}x=\tau_{-K, L}x$. In particular, $x_{k}+l=x_{K}+L$.
\\ \indent More importantly, $\psi$ is nondecreasing: when $x_{i_1}+\omega(k-i_1) + l > x_{i_1}+\omega(K-i_1) + L$, then $\omega(k-K)+l-L>0$ and hence by Proposition \ref{numbertheory} it must hold that $\tau_{K-k,l-L}x>x$, i.e. that $\tau_{-k,l}x> \tau_{-K, L}x$. In particular, $\psi(x_{i_1}+\omega(k-i_1) + l) = x_{k}+l \geq x_{K}+L = \psi(x_{i_1}+\omega(L-i_1) + L)$.
 \\
\indent It is also clear from the definition that $\psi(\xi+1)=\psi(\xi)+1$ at the points where $\psi$ is defined.
\\
\indent These observations imply that $\psi$ can be extended to a nondecreasing map $\psi:\R\to\R$ with $\psi(\xi+1)=\psi(\xi)+1$. We now define the sequence $y$ by
$$y_j := \psi(y^{q/p}_j) = \psi\left(x_{i_1} + \frac{q}{p}\left(j-i_1\right)\right)\ .$$
We remark that $y\in \mathcal{B}_{p,q}$. This follows from the properties of $\psi$, that is 
$$y_{j-k}+l = \psi\left(x_{i_1}+\frac{q}{p}(j-i_1)-\frac{q}{p}k + l\right)  \left\{ \begin{array}{lll} \leq y_j  & \mbox{when} & -\frac{q}{p}k+l<0\ ,  \\ = y_j  & \mbox{when} & -\frac{q}{p}k+l=0\ , \\  \geq y_j  & \mbox{when} & -\frac{q}{p}k+l>0\ .  \end{array} \right.$$
Moreover, $y$ satisfies (\ref{confinementstatement}). This is true because $y_j=\psi(y^{q/p}_j)$, $x_j=\psi(x^{\omega}_j)$ and $(U_{p,q}^ay)_j=\psi((U_{p,q}^ay^{q/p})_j)$. Indeed, because $\psi$ is nondecreasing, it preserves the inequalities that hold for $y^{q/p}_j$, $x^{\omega}_j$ and $(U_{p,q}^ay^{q/p})_j$.
\end{proof}

\noindent We now combine Proposition \ref{pigeonholeproposition} and Theorem \ref{confinement} to prove Theorem \ref{regularitytheorem}.

\begin{proof}[{\bf Of Theorem \ref{regularitytheorem}}] Let $x\in\mathcal{B}_{\omega}$ satisfy the requirement of Theorem \ref{regularitytheorem} and let $a=a(p,q,\omega,i_2-i_1)$. By Theorem \ref{confinement} there is a $y\in \mathcal{B}_{p,q}$ so that
\begin{align}\label{conf}
y_j \leq x_j \leq (U_{p,q}^ay)_j \ \mbox{for all} \ i_1\leq j \leq i_2\ .
\end{align}
In particular, when $i_1 \leq j+mp, j+np\leq i_2$, then
$$y_j = y_{j+mp}-mq \leq x_{j+mp}-mq \leq (U_{p,q}^ay)_{j+mp}-mq =  (U_{p,q}^ay)_j $$
and similarly with $m$ replaced by $n$, so that
\begin{align}\label{recurrenceestimate}
|x_{j+mp}-x_{j+np}-(m-n)q| \leq |(U_{p,q}^ay)_j - y_j| \ \mbox{when}\ i_1 \leq j+mp, j+np\leq i_2\ .
\end{align}
But according to Proposition \ref{pigeonholeproposition}, there exists an $i_0$ so that 
\begin{align}\label{pige}
\sum_{j=i_0-r}^{i_0+r-1} |(U^a_{p,q}y)_j-y_j| \leq \frac{2r\cdot a}{p}\ .
\end{align}
When $i_1+r\leq i_0+np, i_0+mp \leq i_2-r+1$, then we can sum (\ref{recurrenceestimate}) from $j=i_0-r$ to $j=i_0+r-1$ and use estimate (\ref{pige}) to obtain (\ref{pigeintheorem}).
\end{proof}

\section{Destroying periodic foliations}\label{periodicdestructionsection}
In this section, we will prove that periodic minimal foliations of (\ref{recrel}) can be destroyed by an arbitrarily small smooth perturbation of the local potentials. This result is well-known and also contained in \cite{Matherdestruction}. We nevertheless provide it here, both for completeness and because along the way we will derive some estimates that are necessary later, for the study of irrational foliations.
\\
\indent It is tempting to think that it is completely obvious that a ``generic'' periodic action $W_{p}:\X_{p,q}\to\R$ does not support a minimal foliation. This is because the collection of Morse functions $f:\X_{p,q}\to \R$ on each finite-dimensional space $\X_{p,q}$ is open and dense in the $C^k$-topology for any $k\geq 2$ and because a Morse function only possesses isolated stationary points \cite{Hirsch},  \cite{matsumoto}. Nevertheless, one should note that by far not every Morse function on $\X_{p,q}$ is the periodic action of a variational recurrence relation, i.e. not every Morse function is the sum of finite-range local potentials that satisfy conditions {\bf A}-{\bf E}. For this reason, we do not use Morse theory in this section.
\\ \indent 
For what follows it is helpful to define, for a Birkhoff sequence $x$, the set 
$$\Sigma_{x}:=\{ x_{k}+l \ | \ k,l \in \Z\} \subset \R\ .$$ 
In the context of twist maps, $\Sigma_x$ is sometimes referred to as the {\it extended orbit} of $x$. Obviously, $\Sigma_x$ is invariant under the integer shift $\xi\mapsto \xi+1$. \\
\indent In case that $\Sigma_x$ is not a dense subset of $\R$, it admits a nonempty complementary interval $(\xi_-, \xi_+)$ with $\Sigma_{x} \cap (\xi_-, \xi_+) = \emptyset$. Such an interval is sometimes called a {\it gap}.  These gaps will be important when we perturb the local potentials. \\
\indent Not all Birkhoff sequences have an extended orbit that admits a gap. But when $x\in \mathcal{B}_{p,q}$ is periodic, then $\Sigma_x$ is discrete and therefore it certainly has gaps. More precisely, when $x\in\mathcal{B}_{p,q}$ then it is clear that the set $\Sigma_{x}\cap [0,1)$ has a cardinality less than or equal to $p$. By the pigeonhole principle, this means that there exists at least one complementary interval to $\Sigma_{x}$ of length at least $\frac{1}{p}$. We will see later that in certain situations, much larger gaps may even exist.
\\
\indent In Theorem \ref{destructionperiodic} below, these gaps will act as the support of a small periodic ``bumpy'' perturbation. In the following simple lemma, we establish the existence of such {\it periodic bump functions} and measure their smoothness. We omit the proof.
\begin{lemma}\label{phi}
 For every $k \in \N$ there exists a number $0<C_k<\infty$ such that for all real numbers $
 \xi_-$ and $\xi_+$ with $\xi_- < \xi_+ <\xi_- +1$ and every $\varepsilon>0$, there exists a $C^{\infty}$ function $\phi:\R\to \R$ that satisfies
 \begin{itemize}
\item $||\phi||_{C^k}:=\max_{0\leq n\leq k} \sup_{\xi\in \R}\left|\frac{d^n \phi(\xi)}{d \xi^n} \right| \leq \varepsilon$.
\item $\phi(\xi + 1) = \phi(\xi)$ for all $\xi\in \R$.
\item $\phi(\xi) = 0$ for $\xi_+ \leq \xi \leq \xi_- + 1$.
\item $\phi(\xi)>0$ for $\xi_- < \xi < \xi_+$.
\item $\phi(\xi) = \frac{\varepsilon (\xi_+-\xi_-)^k}{C_k}$ for $\xi_-  + \frac{\xi_+-\xi_-}{4}\leq \xi \leq \xi_{+} - \frac{\xi_+ - \xi_-}{4}$.
\end{itemize}
 \end{lemma}
For a given collection of local potentials $S_j$ and a given periodic minimizer $y^{\rm min}\in \mathcal{B}_{p,q}$ for these potentials, let us assume that $(\xi_-, \xi_+)$ is a gap in $\Sigma_{y^{\rm min}}$. Then we let $\phi:\R\to \R$ be the smooth periodic bump function satisfying the conclusions of Lemma \ref{phi} above and we define the local potentials $S_{j}^{\varepsilon}$ by
\begin{align}\label{sepsilondef}
S^{\varepsilon}_j(x):=S_j(x) + \phi(x_j)\ .
\end{align}
\noindent It is easy to check that the $S_j^{\varepsilon}$ satisfy conditions {\bf A}-{\bf E} of Section \ref{problemsetup} when the $S_j$ do. Moreover, $S_j^{\varepsilon}$ is a small perturbation of $S_j$ in the sense that
$$||S_j^{\varepsilon}- S_j||_{C^k} := \max_{0\leq n\leq k} \ \max_{i_1, \ldots, i_n\in \Z} \ \sup_{x\in \R^{\Z}} \left| \p_{i_1, \ldots, i_n}S_j(x)\right| = \max_{0\leq n\leq k} \sup_{\xi\in \R}\left|\frac{d^n\phi(\xi)}{d \xi^n} \right| = ||\phi||_{C^k} \leq \varepsilon\ .$$ 
Most importantly, these $S_j^{\varepsilon}$ do not admit a periodic minimal foliation. This is the content of Theorem \ref{destructionperiodic}:

\begin{theorem}\label{destructionperiodic}
Assume that the $S_j$ are local potentials that satisfy conditions {\bf A}-{\bf E} of Section \ref{problemsetup} and let $\varepsilon>0$ be a perturbation parameter, $k \in \N_{\geq 2}$ a differentiability degree and $(p, q) \in \N\times \Z$ integers. \\
\indent Moreover, let $y^{\rm min}\in \mathcal{B}_{p,q}$ be a periodic minimizer for the local potentials $S_j$ and let $(\xi_-, \xi_+)\subset \R$ be a nonempty maximal complementary interval to $\Sigma_{y^{\rm min}}$, that is 
$$\xi_-, \xi_+\in \Sigma_{y^{\rm min}} \ \mbox{and} \ (\xi_-, \xi_+)\cap \Sigma_{y^{\rm min}} = \emptyset\ .$$ 
Recall that such a complementary interval always exists. Let the local potentials $S^{\varepsilon}_j$ be defined as in (\ref{sepsilondef}), where $\phi=\phi(\xi)$ obeys the conclusions of Lemma \ref{phi}. Then the $S_j^{\varepsilon}$ satisfy conditions {\bf A}-{\bf E} and the estimate $$||S_j^{\varepsilon}- S_j||_{C^k}\leq \varepsilon\ .$$ 
Moreover, 
the following are true.
\begin{itemize}
\item[{\bf 1.}] The periodic minimizer $y^{\rm min}$ of the unperturbed local potentials $S_j$ is also a minimizer of the perturbed local potentials $S_j^{\varepsilon}$. Moreover, when $y\in \X_{p,q}$ is a periodic minimizer of the $S_j^{\varepsilon}$, then $y=\tau_{k,l}(y^{\rm min})$ for some integers $k,l$. 
\item[{\bf 2.}] Let us define, for $M\in \N$, the function 
$$W_{Mp}^{\varepsilon}:\X_{Mp, Mq}\to\R\ \mbox{by}\ W^{\varepsilon}_{Mp}(x):=\sum_{j=1}^{Mp} S_j^{\varepsilon}(x)\ .$$
When $x\in \X_{Mp,Mq}$ satisfies
$$\xi_- +\frac{(\xi_+-\xi_-)}{4} \leq x_{k_i} + l_i  \leq  \xi_+ - \frac{(\xi_+-\xi_-)}{4}$$ 
for certain integers $0\leq k_1 < k_2 < \ldots < k_ N < Mp$ and $l_1, l_2, \ldots, l_N \in \Z$, then
$$W_{Mp}^{\varepsilon}(x)-W_{Mp}^{\varepsilon}(y^{\rm min}) \geq \frac{N \varepsilon  (\xi_+-\xi_-)^k}{C_k}\ .$$
\end{itemize}  
\end{theorem}
\begin{proof}
We have already seen that the $S_j^{\varepsilon}$ satisfy conditions {\bf A}-{\bf E} and that $||S_{j}^{\varepsilon}-S_j||_{C^k}\leq \varepsilon$. \\
\indent Let us define the function $W_p^{\varepsilon}:\X_{p,q}\to \R$ by $W_p^{\varepsilon}(x):=\sum_{j=1}^pS_j^{\varepsilon}(x)$. 
Then it holds for every $x\in \X_{p,q}$ that 
$$W_{p}^{\varepsilon}(x) - W_{p}^{\varepsilon}(y^{\rm min}) = W^0_{p}(x)-W^0_{p}(y^{\rm min}) + \sum_{j=1}^{p} \left( \phi(x_j)-\phi(y^{\rm min}_j) \right) \geq 0\ .$$
The inequality above holds because by assumption, $y^{\rm min}$ is a minimizer of $W_{p}=W_{p}^0$ on $\X_{p,q}$ and because by construction of the function $\phi$ it holds that $\phi(y^{\rm min}_j)=0$ for all $j$ and $\phi(x_j) \geq 0$ for all $j$. This proves the first conclusion in part {\bf 1} of the theorem, i.e. that $y^{\rm min}$ is a $(p, q)$-periodic minimizer of the local potentials $S_j^{\varepsilon}$.\\ 
\indent To prove the second conclusion in part {\bf 1} of the theorem, we let $y\in \mathcal{B}_{p,q}$ be a periodic minimizer. We then consider two possibilities. When $y_0\in \Sigma_{y^{\rm min}}$, then $y_0=y^{\rm min}_k+l$ for some $k,l\in \Z$. Because the collection of $(p,q)$-minimizers is strictly ordered, this means that $y=\tau_{-k,l}y^{\rm min}$, that is $y$ is a translate of $y^{\rm min}$. \\
\indent The other possibility is that $y_0\notin \Sigma_{y^{\rm min}}$. In this case, $y$ is not a translate of $y^{\rm min}$, and hence $U^a_{p,q}y^{\rm min}\ll y\ll U^{a+1}_{p,q}y^{\rm min}$ for some $a\in\Z$. On the other hand, because $\xi_-$ and $\xi_+$ are ``consecutive'' elements of $\Sigma_{y^{\rm min}}$, it holds that $\xi_-=\left(U_{p,q}^{b}y^{\rm min}\right)_0$ and $\xi_+= \left(U_{p,q}^{b+1}y^{\rm min}\right)_0$ for some $b\in \Z$. In particular, $U^{b}_{p,q}y^{\rm min}\ll U_{p,q}^{b-a}y \ll U^{b+1}_{p,q}y^{\rm min}$ and hence, recalling that $U_{p,q}=\tau_{s,t}$, 
$$\xi_-  < y_{(a-b)s}+(b-a)t < \xi_+\ .$$
But this implies that $y$ can not be a $(p,q)$-minimizer. Indeed, we find that
$$W_{p}^{\varepsilon}(y) - W_{p}^{\varepsilon}(y^{\rm min}) = W^0_{p}(y)-W^0_{p}(y^{\rm min}) + \sum_{j=1}^{p} \left( \phi(y_j)-\phi(y^{\rm min}_j) \right) > 0\ ,$$
because by construction of the function $\phi$ it holds that $\phi(y_{j})>0$ whenever $\xi_- < y_j < \xi_+$.
\\
\indent By a similar computation, one proves part {\bf 2} of the theorem. More precisely, for any $x \in \X_{Mp,Mq}$ with $\xi_- +\frac{(\xi_+-\xi_-)}{4} \leq x_{k_i} + l_i  \leq  \xi_+ - \frac{(\xi_+-\xi_-)}{4}$ for $i=1,\ldots, N$, we find that
$$W_{Mp}^{\varepsilon}(x) - W_{Mp}^{\varepsilon}(y^{\rm min}) =W^0_{Mp}(x)-W^0_{Mp}(y^{\rm min}) + \sum_{j=1}^{Mp} \left( \phi(x_j)-\phi(y^{\rm min}_j) \right) \geq \frac{N\varepsilon(\xi_+-\xi_-)^k}{C_k}\ .$$
The inequality now holds because by assumption, $y^{\rm min}$ is a $(p,q)$-periodic minimizer of the $S_j$ and hence by Proposition \ref{periodicminimizersproposition} also an $(Mp, Mq)$-periodic minimizer of the $S_j$, and because by construction of the function $\phi$ it holds that $\phi(y^{\rm min}_j)=0$ for all $j$ and $\phi(x_j) \geq 0$ for all $j$ and $\phi(x_{k_i})=\phi(x_{k_i}+l_i) \geq \frac{\varepsilon(\xi_+-\xi_-)^k}{C_k}$ for every $i=1,\ldots, N$. 
\end{proof}
Part {\bf 1} of Theorem \ref{destructionperiodic} says that for the perturbed potentials $S_j^{\varepsilon}$ there is only one $\tau$-orbit of $(p,q)$-periodic minimizers. Every such group orbit is discrete. Therefore, if the unperturbed $S_j$ had supported a continuous family of periodic minimizers, then this family is destroyed after perturbation. \\
\indent Part {\bf 2} of Theorem \ref{destructionperiodic} measures how much an $x\in \X_{Mp, Mq}$ fails to be an $(Mp, Mq)$-periodic minimizer when it takes values in the middle half of the interval $(\xi_-, \xi_+)$ or one of the integer translates of this interval. We will need the estimate of part {\bf 2} in the next section.

\section{Destroying irrational foliations: the first steps}\label{firststepssection}
For a given collection of local potentials $S_j$ that satisfy conditions {\bf A}-{\bf E} of Section \ref{problemsetup} and a given irrational rotation number $\omega$, we will now start the construction of the perturbations $S_j^{\varepsilon}$ that do not admit a minimal foliation of rotation number $\omega$. \\
\indent The candidate perturbations $S_j^{\varepsilon}$ will be constructed in this section, by a procedure similar to that in \cite{Matherdestruction}. After this section, we will deviate from the ideas of \cite{Matherpeierls} and \cite{Matherdestruction}.\\
\indent The idea is that, when we are given an irrational rotation number $\omega$, we approximate it by a rational number $\frac{q}{p}$, where $p\in \N$ and $q\in \Z$. Let us assume that 
\begin{align}\label{qpomega}
\frac{q}{p}< \omega < \frac{q+1}{p}\ .
\end{align}
Thus, we assume that $\frac{q}{p}$ approximates $\omega$ from below and that it is the best approximation from below with denominator $p$. We will not assume that $p$ and $q$ are relative prime. In case that $\frac{q-1}{p}< \omega < \frac{q}{p}$, the analysis is similar to the case we consider in detail here. \\
\indent When $y^{\rm min}\in \mathcal{B}_{p,q}$ is a periodic minimizer for the unperturbed potentials $S_{j}$, then the set  
$$\Sigma_{y^{\rm min}}=\{y_k^{\rm min} + l \ | k,l\in \Z\} $$ has a gap of length at least $\frac{1}{p}$.  This  was explained in Section \ref{periodicdestructionsection}. Applying Theorem \ref{destructionperiodic}, we can therefore immediately conclude:
\begin{proposition}\label{firstperturbation}
Let $\omega\in \R \backslash \Q$, let $(p,q)\in \N\times \Z$ satisfy (\ref{qpomega}) and let $y^{\rm min}\in \mathcal{B}_{p,q}$ be a periodic minimizer. \\ \indent Then there are, for every $\varepsilon>0$ and every $k\in \N_{\geq 2}$, perturbations $S_j^{\varepsilon,1}$ of the original potentials $S_j$, satisfying conditions {\bf A}-{\bf E} and the estimate $||S^{\varepsilon,1}_j-S_j||_{C^k}\leq \frac{\varepsilon}{3}$, as well as a number $\xi\in \R$, with the following property. \\ \indent When $x\in \X_{Mp,Mq}$ satisfies 
$\xi \leq x_{k_i}+l_i\leq \xi + \frac{1}{2p}$ for certain $0\leq k_1 < k_2 < \ldots < k_ N < Mp $ and $l_1, l_2, \ldots, l_N \in \Z$, then 
$$W_{Mp}^{\varepsilon,1}(x)-W_{Mp}^{\varepsilon,1}(y^{\rm min})\geq \frac{N\varepsilon}{3C_kp^k}\ .$$
\end{proposition} 
Next, we let $p'\in \{2, 3, \ldots\}$ be the unique integer for which 
\begin{align}\label{defpprime}
\frac{1}{p'p} <\omega - \frac{q}{p} < \frac{1}{(p'-1)p}\ .
\end{align}
Such $p'$ exists, because we assumed that $\omega$ is irrational and that $\frac{q}{p}< \omega < \frac{q+1}{p}$. We will now investigate the periodic minimizers of the perturbed action $W_{p'p}^{\varepsilon,1} = \sum_{j=1}^{p'p}S_j^{\varepsilon,1}$ in the space $\X_{p'p, p'q+1}$. Elements of this space have rotation number $\frac{q}{p}+\frac{1}{p'p}$. Because $\frac{q}{p}<\frac{q}{p}+\frac{1}{p'p}<\omega$, this new rotation number is a better rational approximation of $\omega$ than the rational approximation $\frac{q}{p}$ that we started out with. \\
\indent The main result of this section is the following theorem:
\begin{theorem}\label{sepsilon2}
Let $\omega\in\R \backslash \Q$ and let $(p,q)\in \N\times \Z$ satisfy (\ref{qpomega}). Moreover, let $p'$ be defined by (\ref{defpprime}), assume that $r\leq p'p$ and define $C_{k,r}:=12CC_k(2r+1)^2$. \\
\indent For $\varepsilon>0$ and $k\in \N_{\geq 2}$, we let the $S_j^{\varepsilon,1}$ be as in Proposition \ref{firstperturbation} and we let $x^{\rm min}\in \mathcal{B}_{p'p, p'q+1}$ be any periodic minimizer of the $S_j^{\varepsilon,1}$.\\
\indent Then there is a complementary interval to $\Sigma_{x^{\rm min}}$ of length at least 
$$\varepsilon/C_{k,r}p^{k+1}\ .$$
As a consequence, there exists a further perturbation $S_j^{\varepsilon,2}$ of the $S_j^{\varepsilon,1}$ satisfying conditions {\bf A}-{\bf E} and the estimate $||S_j^{\varepsilon,2}-S_j^{\varepsilon,1}||_{C^k}\leq \frac{\varepsilon}{3}$, as well as an $\eta\in \R$ so that the following holds. \\
\indent First of all, $x^{\rm min}$ is also a minimizer for the local potentials $S_j^{\varepsilon, 2}$. Secondly, for every $x\in \X_{p'p, p'q+1}$ with 
$$\eta\leq x_{0} \leq \eta+  \varepsilon/2C_{k,r}p^{k+1}\ ,$$ 
it holds that
\begin{align}\label{idontknowanymorewhattowritehere}
W_{p'p}^{\varepsilon,2}(x)-W_{p'p}^{\varepsilon,2}(x^{\rm min}) \geq  \frac{\varepsilon}{3C_k}\left(\varepsilon/C_{k,r}p^{k+1}\right)^k\ .
\end{align}
\end{theorem}
We will later formulate conditions on $\omega$ and $\frac{q}{p}$ under which estimate (\ref{idontknowanymorewhattowritehere}) is ``unexpectedly strong''. This will be the main point of Theorem \ref{sepsilon2}. \\
\indent
For the proof of Theorem \ref{sepsilon2}, we need a quantitative continuity result that we formulate separately here. In fact, this is the first time we use condition {\bf E} of Section \ref{problemsetup}.
\begin{proposition}[Lipschitz continuity]\label{lipschitzcontinuity} Let $i_1\leq i_2$ be integers and define 
$$W_{[i_1,i_2]}:\R^{\Z}\to \R \ \mbox{by} \ W_{[i_1,i_2]}(x):= \sum_{j=i_1}^{i_2}S_j(x)\ .$$ 
Then $W_{[i_1, i_2]}$ is $l_1$-Lipschitz continuous:
$$|W_{[i_1, i_2]}(x)-W_{[i_1, i_2]}(y)|\leq C(2r+1)||x-y||_{l_1[i_1-r, i_2+r]}\ .$$
Here,
$$||x-y||_{l_1[i_1-r, i_2+r]}:= \! \sum_{j=i_1-r}^{i_2+r} \! |x_j - y_j|\ .$$
\end{proposition}
\begin{proof} We use interpolation:
\begin{align}\nonumber
& |W_{[i_1,i_2]}(x)-W_{[i_1,i_2]}(y)| \leq \sum_{j=i_1}^{i_2} \left| S_j(x)-S_j(y) \right| = 
\sum_{j=i_1}^{i_2} \left| \int_0^1\frac{d}{d\tau} S_j(\tau x+(1-\tau)y)d\tau \right| \\ \nonumber  & \leq 
\sum_{j=i_1}^{i_2} \sum_{|k-j|\leq r} \int_0^1 \left| \p_kS_j(\tau x+(1-\tau)y)\right|d\tau \cdot |x_k-y_k| \leq C(2r+1)\sum_{k=i_1-r}^{i_2+r} |x_k-y_k|\ .
\end{align}
The final inequality follows from changing the order of summation. 
\end{proof}
Now we are ready to prove Theorem \ref{sepsilon2}.\\ \mbox{} \\ \noindent
 {\it Proof:}\indent [{\bf Of Theorem \ref{sepsilon2}}]
Let us assume that $x^{\rm min} \in \X_{p'p, p'q+1}$ and $y^{\rm min}\in \X_{p'p,p'q}$ are periodic minimizers for the potentials $S_j^{\varepsilon,1}$. That is, they minimize the action $W_{p'p}^{\varepsilon,1}$ over the spaces $\X_{p'p, p'q+1}$ and $\X_{p'p,p'q}$ respectively. Moreover, let $\xi\in \R$ be as in the conclusion of Proposition \ref{firstperturbation} and let us assume that 
$$\xi \leq x_{k_i}^{\rm min} + l_i \leq \xi+\frac{1}{2p} \ ,$$  
for certain integers
$0\leq k_1 < k_2 < \ldots < k_N < p'p$ and $l_1, l_2, \ldots, l_N\in \Z$. To start with, we will prove an upper bound for $N$ as follows.
\\
\indent We first define the sequences $\tilde x^{\rm min} \in \X_{p'p, p'q}$ and $\tilde y^{\rm min} \in \X_{p'p, p'q+1}$ by 
\begin{align}
\begin{array}{ll}
\tilde x_i^{\rm min}:= & x_{i}^{\rm min} - n\ , \ \mbox{for} \ np'p \leq i < (n+1)p'p\ , \\
\tilde y_i^{\rm min}:= & y_{i}^{\rm min} + n\ , \ \mbox{for} \ np'p \leq i < (n+1)p'p\ .
\end{array}
\end{align}
It is clear from these definitions that $\tilde x^{\rm min} \in \X_{p'p, p'q}$ and $\tilde y^{\rm min} \in \X_{p'p, p'q+1}$. Thus, by Proposition \ref{firstperturbation}, due to our assumptions on $x^{\rm min}$ and because $\tilde x^{\rm min}_i=x^{\rm min}_i$ for all $0\leq i < p'p$, it holds that
\begin{align}\label{bumpaction}
W_{p'p}^{\varepsilon,1}(\tilde x^{\rm min})-W_{p'p}^{\varepsilon,1}(y^{\rm min})\geq \frac{N\varepsilon}{3C_k p^k}\ .
\end{align}
At the same time, because $x^{\rm min}$ is a minimizer, we have that 
\begin{align}\label{trivialestimate}
\begin{array}{l}
W_{p'p}^{\varepsilon,1}(x^{\rm min})- W_{p'p}^{\varepsilon,1}(\tilde y^{\rm min}) \leq 0\ .
\end{array}
\end{align}
Finally, using that $r\leq p'p$, it is easy to see that  
$$|x_i^{\rm min}-\tilde x_i^{\rm min}| = |y_i^{\rm min}-\tilde y_i^{\rm min}| =\left\{ \begin{array}{ll}  0 & \mbox{when} \ 0 \leq i < p'p \\  1 & \mbox{when} \ -r\leq i <0 \ \mbox{or} \ p'p \leq i \leq p'p+r-1 \end{array} \right. .$$ 
This shows that 
$$||x^{\rm min}-\tilde x^{\rm min}||_{l_1[-r, p'p-1+r]} = ||y^{\rm min}-\tilde y^{\rm min}||_{l_1[-r, p'p-1+r]} =  2r\ .$$ 
By Proposition \ref{lipschitzcontinuity} on Lipschitz-continuity, it therefore holds that
\begin{align}\label{lipschitzaction}
\begin{array}{l}
\left| W_{p'p}^{\varepsilon,1}(x^{\rm min})-W_{p'p}^{\varepsilon,1}(\tilde x^{\rm min})\right| \leq 2r(2r+1)C\ ,\\
\left| W_{p'p}^{\varepsilon,1}(y^{\rm min})-W_{p'p}^{\varepsilon,1}(\tilde y^{\rm min})\right| \leq 2r(2r+1)C\ .
\end{array}
\end{align}
Combining (\ref{bumpaction}), (\ref{trivialestimate}) and (\ref{lipschitzaction}) we obtain in particular that 
\begin{align}\label{almostthere}\nonumber
\frac{N\varepsilon}{3C_k p^k}\leq W_{p'p}^{\varepsilon,1}(\tilde x^{\rm min})-&W_{p'p}^{\varepsilon,1}(y^{\rm min}) =  \left( W_{p'p}^{\varepsilon,1}(\tilde x^{\rm min})-W_{p'p}^{\varepsilon,1}(x^{\rm min})\right) +\\ \left( W_{p'p}^{\varepsilon,1}(x^{\rm min})-W_{p'p}^{\varepsilon,1}(\tilde y^{\rm min}) \right)  &+ \left(W_{p'p}^{\varepsilon,1}(\tilde y^{\rm min})-W_{p'p}^{\varepsilon,1}(y^{\rm min})\right) \leq 4r(2r+1)C\ .\end{align}
This shows that $N$ is bounded from above:
 $$N\leq 12CC_kr(2r+1)p^k/\varepsilon < 6CC_{k,r}(2r+1)^2p^k/\varepsilon \ . $$
Using the pigeonhole principle, we therefore see that the interval $[\xi, \xi+\frac{1}{2p}]$ contains a subinterval of lenght at least equal to 
$$\frac{1/2p}{6CC_k(2r+1)^2p^k/\varepsilon}  = \frac{\varepsilon}{12CC_k(2r+1)^2p^{k+1}}= \frac{\varepsilon}{C_{k,r}p^{k+1}}\, $$ 
that is complementary to $\Sigma_{x^{\rm min}}$. \\ \indent The remaining statements of Theorem \ref{sepsilon2} now follow immediately from Theorem \ref{destructionperiodic}.
$\hfill \square$

\section{A persistence theorem for gaps}\label{periodicpersistencesection}
In the proof of Theorem \ref{sepsilon2} we saw that when $p'p\geq r$ and
$$\frac{q}{p} < \frac{q}{p}+\frac{1}{p'p}< \omega < \frac{q}{p}+\frac{1}{(p'-1)p} \leq \frac{q}{p}+ \frac{1}{p}\ ,$$
then there exist arbitrarily small perturbations $S_{j}^{\varepsilon,1}$ of the potentials $S_j$, for which the extended orbit $\Sigma_{x^{\rm min}}$ of any periodic minimizer $x^{\rm min}\in \mathcal{B}_{p'p, p'q+1}$ has a gap of length $\varepsilon/C_{k,r}p^{k+1}$. This gap acted as the support of a further small perturbation $S_j^{\varepsilon,2}$.\\
\indent In Theorem \ref{thebigresult} below, we formulate three conditions on $\omega$ and $\frac{q}{p}$ that ensure that both $\frac{q}{p}+\frac{1}{p'p}$ and $\frac{q}{p}+\frac{1}{(p'-1)p}$ are extremely close to $\omega$. It turns out that under these conditions, part of the gap in the extended orbit of $x^{\rm min}\in \mathcal{B}_{p'p, p'q+1}$ persists to nearby rotation numbers. That is, the minimizers of the $S_j^{\varepsilon,2}$ with rotation numbers $\Omega$ in a certain open neighborhood of $\omega$, will have the same gap.  \\
\indent The precise statement is as follows:
\begin{theorem}\label{thebigresult}
Let $\omega \in\R \backslash \Q$, $\varepsilon >0$, $k\in \N_{\geq 2}$, $\tau >2$ and $(p,q) \in \N\times \Z$ satisfy the conditions
\begin{itemize}
\item[{\bf A1}] The quotient $\frac{q}{p}$ is a very good approximation of $\omega$ in the sense that 
$$\frac{\gamma}{p^{\tau+1}} \leq \omega-\frac{q}{p}<\frac{\gamma}{p^{\tau}}\ .$$
\item[{\bf A2}] The integer $p\geq 1$ and the real number $\tau>2$ are so large that
\begin{align}\label{nrequirement}
p^{\tau-1-2k(k+1)} \geq  10\gamma  \left(\frac{C_{k,r}}{\varepsilon}\right)^{2+2k} .
 \end{align}
\item[{\bf A3}] For technical reasons, we ask that $\varepsilon \leq C_{k,r}/10$, that $p^{\tau-1}\geq 10\gamma$ and that $p^{\tau-1}\geq r\gamma$.
\end{itemize}
Let $p'$ be defined by (\ref{defpprime}). Then $p'p\geq r$ and we let $S_j^{\varepsilon,2}$ and $\eta\in\R$ satisfy the conclusions of Theorem \ref{sepsilon2}. Furthermore, suppose that $\Omega\in \R$ is chosen so that 
\begin{align}\label{QPestimate}
\frac{q}{p} + \frac{1}{p'p} \leq\Omega\leq \frac{q}{p} + \frac{1}{(p'-1)p}\ .
\end{align}
Then there is no maximally periodic Birkhoff minimizer $x\in \mathcal{B}_{\Omega}$ for the local potentials $S_j^{\varepsilon, 2}$ that satisfies 
$$\eta \leq x_0 \leq \eta+\frac{\varepsilon}{2C_{k,r}p^{k+1}}\ .$$ 
\end{theorem}
The proof of Theorem \ref{thebigresult} will be given in Section \ref{proofsection}. In the remainder of this section, we will show that when $\omega$ is not too irrational, the conditions of Theorem \ref{thebigresult} can actually be satisfied. As a consequence, we can then prove Theorem \ref{maintheorem} of the introduction.
 
\subsection{A class of not so irrational numbers}\label{rationalnumberssection}
One may wonder which rotation numbers admit rational approximations that satisfy the conditions of Theorem \ref{thebigresult}. It turns out convenient to define, for every $\gamma>0$ and $\tau>2$, the following sets
$$\mathcal{L}_{\gamma, \sigma}^- \!:=\! \left\{ \omega\in \R\backslash \Q\ \! \left| \ \! \exists\ \mbox{sequence}\! \ (p_j, q_j)\in \N\times \Z\ \mbox{with} \ 0< \omega - \frac{q_j}{p_j} < \frac{\gamma}{p_j^{\sigma}} \ \mbox{and} \lim_{j\to \infty} p_j=\infty\right. \! \right\} ,$$
$$\mathcal{L}_{\gamma, \sigma}^+ \!:=\! \left\{ \omega\in \R\backslash \Q\ \! \left| \ \! \exists\ \mbox{sequence}\! \ (p_j, q_j)\in \N\times \Z\ \mbox{with} \ \!\! -\frac{\gamma}{p_j^{\sigma}}< \omega - \frac{q_j}{p_j} < 0  \ \mbox{and} \lim_{j\to \infty} p_j=\infty \right. \! \right\} .$$
Obviously, $\mathcal{L}_{\gamma, \sigma}^{-}$ consists of rotation numbers that can be approximated quite well by rational numbers from below, whereas the elements of $\mathcal{L}_{\gamma, \sigma}^+$ admit good rational approximations from above. Finally, we define 
$$\mathcal{L}_{\gamma, \sigma}:= \mathcal{L}_{\gamma, \sigma}^- \cup \mathcal{L}_{\gamma, \sigma}^+\ .$$
\noindent For bookkeeping reasons, we will prove Theorems \ref{maintheorem} and \ref{maintheorem2} under the assumption that $\omega\in\mathcal{L}_{\gamma,\sigma}^-$. When $\omega\in  \mathcal{L}_{\gamma,\sigma}^+$, the proofs are almost identical. \\
\indent In order to obtain some intuition about the ``size'' of $\mathcal{L}_{\gamma, \sigma}$, we provide the following proposition. It shows that $\mathcal{L}_{\gamma, \sigma}$ is simultaneously a ``large'' and a ``small'' subset of $\R$.

\begin{proposition} Every $\mathcal{L}_{\gamma, \sigma}$ contains all Liouville numbers and is hence uncountable. Although every $\mathcal{L}_{\gamma, \sigma}$ also contains some Diophantine numbers, it has zero Lebesgue measure. 
\end{proposition}
\begin{proof} To prove that $\mathcal{L}_{\gamma, \sigma}$ is uncountable, one can note that the intersection of all $\mathcal{L}_{\gamma, \sigma}$ is the well-known set of {\it Liouville numbers}:
$$\mathcal{L}:=\bigcap_{\gamma, \sigma}\mathcal{L}_{\gamma, \sigma} = \! \left\{ \omega\in \R\backslash \Q \ | \ \mbox{for all} \ n\in \N\  \mbox{there are} \ (p, q) \in \N\times\Z\ \mbox{so that}\ \left| \omega - \frac{q}{p} \right| < \frac{1}{p^{n}}\right\} .$$
Of course $\mathcal{L}$ is known to be uncountable, and hence so is every $\mathcal{L}_{\gamma, \sigma}$. \\ 
\indent To prove that $\mathcal{L}_{\gamma, \sigma}$ has zero Lebesgue measure, let us recall that if a number $\omega$ is not Liouville, then it is {\it Diophantine}, that is there are constants $\gamma>0$ and $\sigma>2$ so that
$$\omega\in \mathcal{D}_{\gamma, \sigma}:= \left\{ \omega \in \R \ | \ \left| \omega - \frac{q}{p}\right| \geq \frac{\gamma}{p^{\sigma}} \ \mbox{for all} \ (p,q)\in \N\times \Z \right\} .$$ 
It is clear that $\mathcal{L}_{\gamma, \sigma} \neq \mathcal{L}$ and hence every $\mathcal{L}_{\gamma, \sigma}$ contains some Diophantine numbers. But at the same time, it is not hard to check that
$$\mathcal{D}_{\delta, \tau} \subset \R\backslash (\Q \cup \mathcal{L}_{\gamma, \sigma})\ \mbox{for all} \ \gamma, \delta > 0\ \mbox{and} \ \tau < \sigma .$$
Because it is well known that for all $\tau>2$, the union $\bigcup_{\gamma>0} \mathcal{D}_{\gamma, \tau}$ has full Lebesgue measure in $\R$, this makes it clear that $\mathcal{L}_{\gamma, \sigma}$ has zero Lebesgue measure. 
\end{proof}

\subsection{Proof of Theorem \ref{maintheorem}}
Accepting Theorem \ref{thebigresult}, it is now easy to prove Theorem \ref{maintheorem} presented in the introduction. \\
\indent First of all, we remark that the elements of $\mathcal{L}_{\gamma, \sigma}^-$ are the ones for which Theorem \ref{thebigresult} has a nontrivial meaning:
\begin{proposition}\label{satisfyconditions}
Let the perturbation parameter $\varepsilon>0$, the differentiability degree $k\in \N_{\geq 2}$, the real numbers $\gamma>0$ and $\sigma>1+2k(k+1)$ and the rotation number $\omega\in \mathcal{L}_{\gamma, \sigma}^-$ be given. 
\\ \indent 
Then there exists a real number $\tau\geq \sigma$ and a pair $(p,q)\in \N\times \Z$ for which the conditions {\bf A1}, {\bf A2} and {\bf A3} of Theorem \ref{thebigresult} hold.
\end{proposition}
\begin{proof}
When $\omega\in\mathcal{L}_{\gamma, \sigma}^-$, then there exist arbitrarily large integers $p \in \N$ and $q\in \Z$ for which the estimate 
$$0 < \omega - \frac{q}{p} < \frac{\gamma}{p^{\sigma}}$$ 
holds. Clearly, for every such $p$ and $q$ there exists a $\tau\geq \sigma$ so that condition {\bf A1} holds, that is for which
\begin{align}\label{basicpqestimate}
\frac{\gamma}{p^{\tau+1}} \leq  \omega - \frac{q}{p}  < \frac{\gamma}{p^{\tau}}\ .
\end{align}
When $\sigma>1+2k(k+1)$, then one can choose $p$ so large that 
$$p^{\tau-1-2k(k+1)} \geq p^{\sigma-1-2k(k+1)} \geq  10\gamma  \left(\frac{C_{k,r}}{\varepsilon}\right)^{2+2k}\ .$$
This means that condition {\bf A2} holds.\\
\indent 
By choosing $p$ even larger when necessary, we can also make it satisfy condition {\bf A3}. 
\end{proof}
Not surprisingly, when $\omega\in \mathcal{L}_{\gamma, \sigma}^+$, then Proposition \ref{satisfyconditions} is true when condition {\bf A1} is replaced by the estimate
$$-\frac{\gamma}{p^{\tau}} < \omega-\frac{q}{p} \leq -\frac{\gamma}{p^{\tau+1}}\ .$$
Theorem \ref{thebigresult} and Proposition \ref{satisfyconditions} together imply:
\begin{theorem}\label{maintheoremprecise}
Assume that the $S_j$ are local potentials that satisfy conditions {\bf A}-{\bf E} of Section \ref{problemsetup} and let $\varepsilon>0$ be a perturbation parameter, $k\in \N_{\geq 2}$ a differentiability degree, $\gamma>0$ and $\sigma>1+2k(k+1)$ real numbers and $\omega\in \mathcal{L}^-_{\gamma, \sigma}$ a rotation number.\\
\indent Then there exist local potentials $S^{\varepsilon}_j$ that satisfy conditions {\bf A}-{\bf E} and the estimate 
$$||S_j^{\varepsilon} - S_j||_{C^k}\leq \frac{2}{3}\varepsilon\leq \varepsilon, $$
as well as a $\delta >0$ and a nonempty interval $(\eta_-, \eta_+)\subset \R$, for which the following holds. \\
\indent When $\Omega\in \R$ is a rotation number satisfying $\left|\omega- \Omega\right|<\delta$ and  $x\in \mathcal{B}_{\Omega}$ is a maximally periodic global minimizer of the perturbed potentials $S_{j}^{\varepsilon}$, then
 $$x_0\notin (\eta_-, \eta_+)\ .$$
\end{theorem}
\begin{proof}
Given $\varepsilon>0$, our assumptions on $k, \sigma$ and $\omega$ and Proposition \ref{satisfyconditions} guarantee that there exist a $\tau \geq \sigma$ and a pair $(p,q)\in \N\times \Z$ for which the assumptions of Theorem \ref{thebigresult} hold. We fix such $p$ and $q$. \\
\indent We then define $p'$ by (\ref{defpprime}) and consecutively construct the local potentials $S_j^{\varepsilon,2}$, as in Section \ref{firststepssection}. They satisfy conditions {\bf A}-{\bf E} and $||S_j^{\varepsilon,2}-S_j||_{C^k}\leq \frac{2}{3}\varepsilon\leq \varepsilon$ and come with an $\eta\in \R$ for which the conclusion of Theorem \ref{thebigresult} holds. \\
\indent Then it is clear that the conclusions of Theorem \ref{maintheoremprecise} hold with the choices $S_j^{\varepsilon}=S_j^{\varepsilon,2}$, $\eta_-=\eta$, $\eta_+=\eta+\frac{\varepsilon}{2C_{k,r}p^{k+1}}$ and $\delta = \min\left\{\left( \frac{q}{p} + \frac{1}{(p'-1)p}\right)-\omega, \omega-\left(\frac{q}{p}+\frac{1}{p'p}\right)\right\}$.
\end{proof}
Together with a similar theorem for $\omega\in \mathcal{L}_{\gamma, \sigma}^+$, this proves Theorem \ref{maintheorem} in the introduction.

\section{Proof of Theorem \ref{thebigresult}}\label{proofsection}
In this section, we prove Theorem \ref{thebigresult}. But let us first try to provide some intuition. \\
\indent 
In Theorem \ref{sepsilon2} we have seen that an $x\in \X_{p'p, p'q+1}$ with $\eta\leq x_0\leq \eta+\varepsilon/2C_{k,r}p^{k+1}$ will fail to be a minimizer for the perturbed local potentials $S_j^{\varepsilon,2}$ by an amount 
\begin{align}\label{bla1}
\frac{\varepsilon}{3C_k}\left(\varepsilon/C_{k,r}p^{k+1}\right)^k\ .
\end{align}
Admittedly, this seems very little. \\
\indent At the same time, when $\Omega\in\R$ is a rotation number satisfying
$$\frac{q}{p} + \frac{1}{p'p} \leq \Omega \leq \frac{q}{p} + \frac{1}{(p'-1)p}\ ,$$
then clearly 
$$\left| (p'q+1) - \Omega p'p  \right| \leq \frac{1}{p'-1}\ .$$
Therefore, our near-periodicity Theorem \ref{regularitytheorem} says that, at least if $p'p$ and $p'q+1$ were relative prime, any maximally periodic $x \in \mathcal{B}_{\Omega}$ will fail to be $(p'p, p'q+1)$-periodic by an amount of the order
\begin{align}\label{bla2}
\frac{2r}{p'p}\left\lceil p'p \cdot \frac{1}{p'-1} \right\rceil \ .
\end{align}
We will show that when conditions {\bf A1}, {\bf A2} and {\bf A3} of Theorem \ref{thebigresult} hold, then (\ref{bla1}) is much larger than (\ref{bla2}). This simple observation is the core of the proof of Theorem \ref{thebigresult}.
\\ \indent 
Because $p'p$ and $p'q+1$ are not necessarily relative prime, we start with a proposition that says that they almost are:
\begin{proposition}\label{propertiesptilde}
Let us write
$$p'p = m \tilde p \ \ \mbox{and}\ \ p'q+1=m\tilde q\  \mbox{with} \ \tilde p\in \N\ \mbox{and} \ \tilde q\in \Z\ \mbox{relative prime and}\ m\in\N.$$ 
Then the least common multiple of $p$ and $\tilde p$ is $p'p$. Moreover, $p' \leq \tilde p \leq p'p$. 
\end{proposition}
\begin{proof}
It is clear that $p'p$ is a common multiple of $p$ and $\tilde p$. When $n=\alpha p=\beta \tilde p$ is another common multiple, then $\frac{n}{p'p}=n\left(\frac{\tilde q}{\tilde p} - \frac{q}{p}\right) =\beta \tilde q-\alpha q$. Thus $n$ is a multiple of $p'p$.\\
\indent It is obvious that $\tilde p =\frac{p'p}{m}\leq p'p$. The upper bound follows because our definitions say that $\frac{\tilde q}{\tilde p}> \frac{q}{p}$. This implies that $p\tilde q \geq q\tilde p +1$, and therefore that
$$\frac{q}{p} + \frac{1}{\tilde p p} = \frac{q\tilde p+1}{p\tilde p} \leq \frac{p\tilde q}{p \tilde p} = \frac{\tilde q}{\tilde p} = \frac{q}{p}+\frac{1}{p'p} <\omega\ . $$
This proves that $\tilde p \geq p'$, by definition of $p'$.
\end{proof}
In particular, Proposition \ref{propertiesptilde} says that $\tilde p$ is more or less as large as $p'p$. Bearing this in mind, we start the proof of Theorem \ref{thebigresult}.\\ \mbox{} \\ \noindent
{\it Proof:}\indent 
[{\bf Of Theorem \ref{thebigresult}}] 
We start with a trivial estimate. Namely, from assumption {\bf A1} that says that
$\frac{\gamma}{p^{\tau+1}}  \leq \omega - \frac{q}{p} < \frac{\gamma}{p^{\tau}}$ and the definition of $p'$ that says that $\frac{1}{p'p}< \omega - \frac{q}{p} < \frac{1}{(p'-1)p}$ it follows immediately that $p'$ is very large:
\begin{align}\label{bothestimates1}
\frac{p^{\tau-1}}{\gamma} < p' < 1 + \frac{p^{\tau}}{\gamma} \ .
\end{align}
By assumption {\bf A3} it therefore follows that $p'p\geq \frac{p^{\tau}}{\gamma}\geq \frac{p^{\tau-1}}{\gamma} \geq r$. By Theorem \ref{sepsilon2} we can hence construct the perturbed potentials $S^{\varepsilon,2}_j$. They come with an $\eta\in \R$ for which the conclusions of Theorem \ref{sepsilon2} hold. More precisely, when $x^{\rm min} \in \X_{p'p, p'q+1}$ is a periodic minimizer and $x\in \X_{p'p, p'q+1}$ satisfies $\eta \leq x_0\leq\eta+\varepsilon/2C_{k,r}p^{k+1}$, then one has the estimate $W_{p'p}^{\varepsilon,2}(x)-W_{p'p}^{\varepsilon,2}(x^{\rm min}) \geq  \frac{\varepsilon}{3C_k}\left(\varepsilon/C_{k,r}p^{k+1}\right)^k$. In particular, by Proposition \ref{periodicminimizersproposition} such an $x$ can not be a minimizer. \\
\indent We now let $\Omega$ be rotation number that satisfies (\ref{QPestimate}) and we assume that $x \in \mathcal{B}_{\Omega}$ is a maximally periodic Birkhoff minimizer of rotation number $\Omega$ for the $S_j^{\varepsilon,2}$ that satisfies the estimates $\eta \leq x_0\leq \eta+\varepsilon/2C_{k,r}p^{k+1}$. We will show that these assumptions lead to a contradiction.  \\
\indent To do this, we choose an integer $N \geq 30$ with
 \begin{align}\label{choiceN2} 
  3\cdot \frac{C_{k,r}}{\varepsilon} \left( \frac{C_{k,r}p^{k+1}}{\varepsilon}\right)^{k} \leq N \leq 3\cdot\frac{11}{10}\cdot \frac{C_{k,r}}{\varepsilon}\left( \frac{C_{k,r}p^{k+1}}{\varepsilon}\right)^{k}\ .
\end{align}
Because we required that $C_{k,r}/\varepsilon \geq10$ and because $k\geq 0$, it is clear that such $N$ exists. 
\\
\indent Next, we recall from Proposition \ref{propertiesptilde} that we defined $p'p=m\tilde p$ and $p'q+1=m\tilde q$, with $\tilde p$ and $\tilde q$ relative prime. We continue the proof by selecting an integer $i_0$ with $-\tilde p < i_0 \leq 0$ and
with the property that 
\begin{align} \label{basicclosenessestimate2}
|| \tau^{-n}_{\tilde p, \tilde q}x - \tau^{-m}_{\tilde p, \tilde q} x||_{l_1[i_0-r, i_0+r-1]} <  \frac{10}{4} \frac{N\cdot\gamma\cdot r}{p^{\tau-1}} \ \mbox{for all} \ 0\leq m, n \leq N\ .
\end{align}
Using the near-periodicity Theorem \ref{regularitytheorem}, we will now show that such an integer exists.  
\begin{proposition} \label{i0existence} When $p$ and $q$ satisfy assumptions {\bf A1} and {\bf A3},  $N\geq 30$ and $x\in \mathcal{B}_{\Omega}$ is maximally periodic, then there exists an integer $i_0$ with $-\tilde p< i_0 \leq 0$ for which (\ref{basicclosenessestimate2}) holds.
\end{proposition}
\begin{proof}
Because $x$ is maximally periodic, we can apply Theorem \ref{regularitytheorem}. We choose $i_1=-2\tilde p$ and $i_2=(N+1)\tilde p$, so that $i_2-i_1 = (N+3)\tilde p$. 
\\ \indent 
Recall that by Proposition \ref{propertiesptilde} it holds that $p' \leq \tilde p \leq p'p$ and hence by 
(\ref{bothestimates1}) that 
\begin{align}\label{bothestimates2}
\frac{p^{\tau-1}}{\gamma} < \tilde p < p\left(1+\frac{p^{\tau}}{\gamma}\right)\ .
\end{align}
Using assumption {\bf A3} we therefore have that $\tilde p > \frac{p^{\tau-1}}{\gamma}\geq r$ and it apparently holds that $i_1+r < -\tilde p$ and $N\tilde p < i_2-r$. Therefore, Theorem \ref{regularitytheorem} says that there exists an $i_0$ with $-\tilde p < i_0 \leq 0$ so that for all $0\leq m,n\leq N$, it holds that 
\begin{align}
|| \tau^{-n}_{\tilde p, \tilde q}x - \tau^{-m}_{\tilde p, \tilde q} x||_{l_1[i_0-r, i_0+r-1]}  \leq \frac{2r}{\tilde p} a(\tilde p, \tilde q, \Omega, (N+3)\tilde p)\ .
\end{align}
Thus, it remains to show that
\begin{align}\label{estimate2}
\frac{a(\tilde p, \tilde q, \Omega, (N+3)\tilde p)}{\tilde p}= \frac{1}{\tilde p}\left\lceil (N+3)\tilde p\left| \tilde q - \Omega \tilde p\right| \right\rceil<   \frac{5}{4} \frac{N\cdot\gamma}{p^{\tau-1}}.
 \end{align}
Let us first estimate $\left| \tilde q - \Omega\tilde p \right|$. From (\ref{QPestimate}) we have $\frac{\tilde q}{\tilde p}< \Omega < \frac{\tilde q}{\tilde p} + \frac{1}{p'(p'-1)p}$. 
Therefore, 
 \begin{align}\label{qp-QP}
 \left| \tilde q - \Omega \tilde p \right| \leq \frac{\tilde p}{p'(p'-1)p} \leq \frac{1}{p'-1} < \frac{11}{10}\frac{1}{p'} <  \frac{11}{10} \frac{\gamma}{p^{\tau-1}}\ .
 \end{align}
The first estimate in (\ref{qp-QP}) holds because $\tilde p\leq p'p$, see Proposition \ref{propertiesptilde}. 
The second estimate in (\ref{qp-QP}) follows from (\ref{bothestimates1}) and assumption {\bf A3} that together guarantee that $p'>\frac{p^{\tau-1}}{\gamma} \geq 10$, so that $p'-1 > \frac{10}{11}p'$. The final estimate follows from (\ref{bothestimates1}).\\ 
 \indent
Now we can estimate:
\begin{align}\nonumber
 \frac{1}{\tilde p} & \left\lceil (N+3)\tilde p\left| \tilde q - \Omega \tilde p\right| \right\rceil  
 \leq \frac{1}{\tilde p}  +  (N+3)\left| \tilde q-\Omega \tilde p \right|  
 \\ \nonumber & <
 \frac{\gamma}{p^{\tau-1}}  + (N+3)\frac{11}{10}\frac{\gamma}{p^{\tau-1}} < \frac{5}{4} \frac{N\cdot\gamma}{p^{\tau-1}} \ .
 \end{align}
The first inequality follows from the definition of the ceiling function. The second inequality follows estimates (\ref{bothestimates2}) and (\ref{qp-QP}). The final inequality follows because $N\geq 30$. 
\end{proof}
Assumption {\bf A2}, see (\ref{nrequirement}), can also be written as 
\begin{align}\label{A2alternative}
\frac{\gamma}{p^{\tau-1}}\leq \frac{1}{10} \left(\frac{\varepsilon}{C_{k,r}}\right)^2\left( \frac{\varepsilon}{C_{k,r}p^{k+1}}\right)^{2k}\ .
\end{align}
Combining estimates (\ref{choiceN2}), (\ref{basicclosenessestimate2}) and (\ref{A2alternative}), we therefore find that for our choice of $N$ and $i_0$ it holds that
\begin{align}\label{basicclosenessestimate}
|| \tau^{-n}_{\tilde p, \tilde q}x - \tau^{-m}_{\tilde p, \tilde q} x||_{l_1[i_0-r, i_0+r-1]} <  \frac{r\cdot \varepsilon}{C_{k,r}}\left( \frac{\varepsilon}{C_{k,r}p^{k+1}}\right)^{k} \ \mbox{for all} \ 0\leq m, n \leq N\ .
\end{align}
\noindent Having chosen $N$ and $i_0$ in this way, the remainder of the proof will be based on a close investigation of $x$ over the long segment 
$$[i_0, i_0+N\tilde p -1] = \bigcup_{n=0}^{N-1}[i_0+n\tilde p, i_0+(n+1)\tilde p-1]\ .$$
\noindent In fact, we will separately investigate the following two possibilities. Below, we let $x^{\rm min}\in \mathcal{B}_{\tilde p, \tilde q}=\mathcal{B}_{p'p, p'q+1}$ be any $(\tilde p, \tilde q)$-periodic minimizer. 

\begin{itemize}
\item[{\bf 1.}] It may hold that $x$ has a quite large action on all short subsegments. That is, for all $n$ with $0\leq n\leq N-1$ we have 
$$W^{\varepsilon, 2}_{[i_0+n\tilde p, i_0+(n+1)\tilde p-1]}(x) - W^{\varepsilon, 2}_{\tilde p}(x^{\rm min}) \geq  \frac{\varepsilon}{6C_k}\left( \frac{\varepsilon}{C_{k,r}p^{k+1}}\right)^k \ .$$

\item[{\bf 2.}] The other option is that $x$ has a relatively small action on one of the short subsegments. That is, there is an $0\leq n^*\leq N-1$ so that 
$$W^{\varepsilon, 2}_{[i_0+n^*\tilde p, i_0+(n^*+1)\tilde p-1]}(x) - W^{\varepsilon, 2}_{\tilde p}(x^{\rm min}) <   \frac{\varepsilon}{6C_k}\left( \frac{\varepsilon}{C_{k,r}p^{k+1}}\right)^k \ .$$
\end{itemize}
We will analyze these two cases separately now.
 \subsubsection*{Case 1}
In this case,
\begin{align}\nonumber
W^{\varepsilon, 2}_{[i_0, i_0+N\tilde p-1]}(x)-W^{\varepsilon, 2}_{[i_0, i_0+N\tilde p-1]}(x^{\rm min})  = & \sum_{n=0}^{N-1} \left( W^{\varepsilon,2}_{[i_0+n\tilde p, i_0+(n+1)\tilde p-1]}(x) - W^{\varepsilon, 2}_{\tilde p}(x^{\rm min})\right) \\  \geq \frac{N\varepsilon}{6C_k}\left( \frac{\varepsilon}{C_{k,r}p^{k+1}}\right)^k & > 12Cr(2r+1)  \ .\label{orderoneaction}
\end{align}
The final inequality in (\ref{orderoneaction}) holds because of our choice of $N$ in (\ref{choiceN2}), the definition of $C_{k,r}=12CC_k(2r+1)^2$ and the fact that $(2r+1)>2r$. \\
\indent We conclude that in case 1, $W^{\varepsilon,2}_{[i_0, i_0+N\tilde p-1]}(x)$ is much larger than $W^{\varepsilon,2}_{[i_0, i_0+N\tilde p-1]}(x^{\rm min})$. Of course, this does not mean that $x$ is not a global minimizer, because $x^{\rm min}$ is not a finite support variation of $x$. So let us try and compare $x$ to a finite support variation that looks like $x^{\rm min}$ on the very long segment. More precisely, we define $\tilde x \in \R^{\Z}$ by
$$\tilde x_j:= \left\{ \begin{array}{ll} x_j^{\rm min} & \mbox{if}\ i_0 + r\leq j \leq i_0+ N\tilde p -r -1\\
x_{j} & \mbox{otherwise} \end{array} \right. \ .$$
Clearly, $\tilde x$ is a variation of $x$ with support in $[i_0+r, i_0+ N\tilde p -r-1]$. 
We will show that $W^{\varepsilon, 2}_{[i_0,i_0+N\tilde p-1]}(x)-W^{\varepsilon, 2}_{[i_0,i_0+N\tilde p-1]}(\tilde x) >0$. This means that $x$ is not a global minimizer. \\
\indent In view of (\ref{orderoneaction}), it would therefore be useful to have an estimate on 
$$\left| W^{\varepsilon, 2}_{[i_0, i_0+N\tilde p-1]}(\tilde x)-W^{\varepsilon, 2}_{[i_0, i_0+N\tilde p-1]}(x^{\rm min})\right| \ .$$ 
Such an estimate is not hard to obtain, because the definition of $\tilde x$ implies that $\tilde x = x^{\rm min}$ on $[i_0+r, i_0+N\tilde p -r-1]$. This implies that
$$||\tilde x - x^{\rm min}||_{l_1[i_0-r, i_0+N\tilde p + r-1]} = ||x-x^{\rm min}||_{l_1[i_0-r, i_0+r-1]} + ||x-x^{\rm min} ||_{l_1[i_0+N\tilde p-r, i_0+N\tilde p +r-1]}\ .$$
\noindent We will now compute these norms. In fact, we can assume without loss of generality that $|x_0-x_0^{\rm min}| \leq \frac{1}{2}$, so that by (\ref{poincareestimate}) we know that
\begin{align}\nonumber
&|x_j-x_j^{\rm min}| = \left| \left(x_j- x_0 - \Omega\cdot j\right) - \left(x_j^{\rm min}-x_0^{\rm min}-\frac{\tilde q}{\tilde p}j\right) + (x_0-x_0^{\rm min})+ \left(\Omega-\frac{\tilde q}{\tilde p}\right) j\right| \\ 
\leq  & \left| x_j-x_0-\Omega\cdot j\right| + \left| x_j^{\rm min}-x_0^{\rm min} -\frac{\tilde q}{\tilde p}j \right| + \left| x_0-x_0^{\rm min} \right| +\left|\frac{\tilde q}{\tilde p} -\Omega \right| \cdot |j|  \leq \frac{5}{2}+\left|\frac{\tilde q}{\tilde p} - \Omega\right| \cdot |j|\ . \nonumber
\end{align}
Let now $j\in \Z$ be so that $i_0 -r\leq j \leq i_0+r-1$ or $i_0+N\tilde p -r \leq j \leq i_0+N\tilde p-1+r$. Using that $-\tilde p < i_0 \leq 0$ and that by (\ref{bothestimates2}) and assumption {\bf A3} it holds that $\tilde p > \frac{p^{\tau-1}}{\gamma}\geq r$, we find that for such $j$ it holds that $-2\tilde p < j < (N+1)\tilde p$.
Thus, for such $j$ we can estimate 
$$|x_j-x_j^{\rm min}| \leq \frac{5}{2} + \left|\tilde q - \Omega \tilde p\right| (N+1) < 3 \ ,$$
the last inequality following from the estimates (\ref{choiceN2}), (\ref{estimate2}) and (\ref{A2alternative}) and the assumption that $\frac{\varepsilon}{C_{k,r}}\leq \frac{1}{10}$. 
We conclude that
$$||\tilde x - x^{\rm min}||_{l_1[i_0-r, i_0+N\tilde p + r-1]} = ||x-x^{\rm min}||_{l_1[i_0-r, i_0-1]} + ||x-x^{\rm min} ||_{l1[i_0+N\tilde p, i_0+N\tilde p +r-1]} < 3\cdot 4r. $$
Therefore, by Proposition \ref{lipschitzcontinuity} on Lipschitz continuity it holds that 
\begin{align}\label{otheractiondifference}
\left|W^{\varepsilon,2}_{[i_0,i_0+N\tilde p-1]}(\tilde x)-W^{\varepsilon,2}_{[i_0, i_0+N\tilde p-1]}(x^{\rm min})\right| < 12Cr(2r+1)\ .
\end{align}
Combining (\ref{orderoneaction}) and (\ref{otheractiondifference}), we obtain that 
\begin{align}\nonumber
W^{\varepsilon,2}_{[i_0,i_0+N\tilde p-1]}(x)&-W^{\varepsilon,2}_{[i_0,i_0+N\tilde p-1]}(\tilde x) = \\
\left(W^{\varepsilon,2}_{[i_0,i_0+N\tilde p-1]}(x)-W^{\varepsilon,2}_{[i_0,i_0+N\tilde p-1]}(x^{\rm min})\right)  & + \left(W^{\varepsilon,2}_{[i_0,i_0+N\tilde p-1]}(x^{\rm min})-W^{\varepsilon,2}_{[i_0,i_0+N\tilde p-1]}(\tilde x)\right) > 0\ . \nonumber
\end{align}
\noindent This means that $x$ is not a global minimizer. 

\subsubsection*{Case 2}
This is the more subtle case. By the shift-invariance of the potentials, $W^{\varepsilon, 2}_{[i_0, i_0+\tilde p -1]}(\tau^{-n^*}_{\tilde p, \tilde q}x) = W^{\varepsilon, 2}_{[i_0+n^*\tilde p, i_0+(n^*+1)\tilde p -1]}(x)$, so case 2 can also be formulated as 
\begin{align}\label{case2estimatetranslated}
W^{\varepsilon, 2}_{[i_0, i_0+\tilde p -1]}(\tau^{-n^*}_{\tilde p, \tilde q}x)- W^{\varepsilon, 2}_{\tilde p}(x^{\rm min})  < \frac{1}{2}\frac{\varepsilon}{3C_k}\left( \frac{\varepsilon}{C_{k,r}p^{k+1}}\right)^k.
\end{align}
This inspires us to change $x$ on a short segment. In fact, we define $\tilde x \in \R^{\Z}$ by
$$\tilde x_j:= \left\{ \begin{array}{ll}  (\tau^{-n^*}_{\tilde p, \tilde q} x)_j = x_{j+n^*\tilde p} - n^*\tilde q & \mbox{if}\ i_0+r\leq j \leq i_0+\tilde p-r-1 \\
x_{j} & \mbox{otherwise} \end{array} \right. \ .$$
Clearly, $\tilde x$ is a variation of $x$ with support in $[i_0+r, i_0+\tilde p-r-1]$. We will show that $W^{\varepsilon,2}_{[i_0, i_0+\tilde p-1]}(x) - W^{\varepsilon,2}_{[i_0, i_0+\tilde p-1]}(\tilde x)>0$, meaning that $x$ is not a global minimizer.\\
\indent To prove this, we need to make several estimates. We first estimate 
$$\left|W_{[i_0, i_0+\tilde p-1]}^{\varepsilon, 2}(\tilde x)- W_{[i_0, i_0+\tilde p-1]}^{\varepsilon, 2}(\tau^{-n^*}_{\tilde p, \tilde q}x)\right|\ .$$
In fact, the definition of $\tilde x$ implies that $\tilde x_j= (\tau^{-n^*}_{\tilde p, \tilde q}x)_j$ for all $i_0+r \leq j \leq i_0+\tilde p -r -1$ and otherwise $\tilde x_j=x_j$.  Therefore,
\begin{align}
|| \tilde x - \tau^{-n^*}_{\tilde p, \tilde q} & x ||_{l_1[i_0-r, i_0 +\tilde p+r-1]} = || x - \tau^{-n^*}_{\tilde p, \tilde q} x||_{l_1[i_0-r, i_0+r-1]} + || x - \tau^{-n^*}_{\tilde p, \tilde q} x ||_{l_1[i_0+\tilde p-r, i_0+\tilde p + r -1]}  \nonumber \\ 
&= || x-  \tau^{-n^*}_{\tilde p, \tilde q} x  ||_{l_1[i_0-r, i_0+r-1]} + ||\tau^{-1}_{\tilde p, \tilde q} x - \tau^{-(n^*+1)}_{\tilde p, \tilde q} x  ||_{l_1[i_0-r, i_0+r-1]}  \ . 
\nonumber
\end{align}
Because $0\leq n^*\leq N-1$, estimate (\ref{basicclosenessestimate}) applies to both terms and we have that
$$|| \tilde x - \tau^{-n^*}_{\tilde p, \tilde q} x ||_{l_1[i_0-r, i_0 +\tilde p+r-1]}   \leq   \frac{2r\cdot \varepsilon}{C_{k,r}} 
\left( \frac{\varepsilon}{C_{k,r}p^{k+1} } \right)^k .
$$
Hence it follows from Lipschitz continuity and the definition $C_{k,r}=12 C C_k(2r+1)^2$ that
\begin{align}\label{xtildeestimate}
\left| W^{\varepsilon,2}_{[i_0, i_0+\tilde p-1]}(\tilde x)- W^{\varepsilon,2}_{[i_0, i_0+\tilde p -1]}(\tau_{\tilde p,\tilde q}^{-n^*} x) \right| < \frac{1}{4} \frac{\varepsilon}{3C_k}\left( \frac{\varepsilon}{C_{k,r}p^{k+1}}\right)^k.
\end{align}
We continue the proof by defining one more sequence $\hat x \in \R^{\Z}$. It is the $(\tilde p, \tilde q)$-periodic extension of $x|_{[i_0, i_0+\tilde p-1]}$, that is 
\begin{align}\nonumber
 \hat x_j:=   (\tau^m_{\tilde p, \tilde q}x)_{j}=x_{j-m\tilde p} +m\tilde q\ \mbox{when}\ i_0+ m\tilde p \leq j \leq i_0+ (m+1)\tilde p -1\ .\end{align}
It is clear that $\hat x\in \X_{\tilde p, \tilde q}\subset \X_{p'p, p'q+1}$. We will provide two estimates for $\hat x$. First of all, 
because $\hat x_j=x_j$ for $i_0 \leq j \leq i_0+\tilde p -1$ and because 
$i_0 \leq 0< i_0+\tilde p$, it holds that $\hat x_0 = x_0$. This implies, by Theorem \ref{sepsilon2} and because $\eta \leq x_0\leq \eta+\varepsilon/2C_{k,r}p^{k+1}$, that
\begin{align}\label{otherhatxestimate}
W_{[i_0, i_0+\tilde p-1]}^{\varepsilon,2}(\hat x)-W_{\tilde p}^{\varepsilon, 2}(x^{\rm min})=W_{\tilde p}^{\varepsilon,2}(\hat x)-W_{\tilde p}^{\varepsilon, 2}(x^{\rm min})\geq \frac{\varepsilon}{3C_k}\left(\frac{\varepsilon}{C_{k,r}p^{k+1}}\right)^k.
\end{align}
The second estimate for $\hat x$ is similar to (\ref{xtildeestimate}). The key observation is that 
\begin{align}\nonumber
& ||\hat x-x||_{l_1[i_0-r, i_0+\tilde p+r-1]} =  ||\tau^{-1}_{\tilde p, \tilde q} x-x||_{l_1[i_0-r, i_0-1]} + ||\tau^{+1}_{\tilde p, \tilde q} x-x||_{l_1[i_0+\tilde p, i_0+\tilde p+r-1]} \\
\nonumber & = ||\tau^{-1}_{\tilde p, \tilde q}x-x||_{l_1[i_0-r, i_0-1]} + ||x - \tau^{-1}_{\tilde p, \tilde q}x||_{l_1[i_0, i_0+r-1]} = ||x-\tau^{-1}_{\tilde p, \tilde q}x||_{l_1[i_0-r, i_0+r-1]} \ . 
\end{align}
With the help of (\ref{basicclosenessestimate}) and Proposition \ref{lipschitzcontinuity} this leads to the estimate
\begin{align}\label{hatxestimate}
 \left| W^{\varepsilon,2}_{[i_0, i_0+\tilde p-1]}(x)- W^{\varepsilon,2}_{[i_0, i_0+\tilde p -1]}(\hat x) \right| < \frac{1}{8} \frac{\varepsilon}{3C_k}\left( \frac{\varepsilon}{C_{k,r}p^{k+1}}\right)^k.
 \end{align}
Combining estimates (\ref{case2estimatetranslated}), (\ref{xtildeestimate}), (\ref{otherhatxestimate}) and (\ref{hatxestimate}), we now conclude that 
\begin{align}\nonumber
W^{\varepsilon,2}_{[i_0, i_0+\tilde p-1]}(x) & - W^{\varepsilon,2}_{[i_0, i_0+\tilde p-1]}(\tilde x) = \\ \nonumber
 \left( W^{\varepsilon,2}_{[i_0, i_0+\tilde p-1]}(x) - W^{\varepsilon,2}_{[i_0, i_0+\tilde p -1]}(\hat x) \right) & + \left( W^{\varepsilon,2}_{[i_0, i_0+\tilde p-1]}(\hat x) - W^{\varepsilon,2}_{\tilde p}(x^{\rm min}) \right) +  \\  
\left( W^{\varepsilon,2}_{\tilde p}(x^{\rm min}) - W^{\varepsilon,2}_{[i_0, i_0+\tilde p-1]}(\tau^{n^*}_{\tilde p, \tilde q} x) \right)  & +  \left( W^{\varepsilon,2}_{[i_0, i_0+\tilde p-1]}(\tau^{n^*}_{\tilde p, \tilde q} x) - W^{\varepsilon,2}_{[i_0, i_0+\tilde p-1]}(\tilde x) \right)  \geq \nonumber \\
  - \! \frac{1}{8} \frac{\varepsilon}{3C_k}\!\left( \frac{\varepsilon}{C_{k,r}p^{k+1}}\right)^k\! +\! \frac{\varepsilon}{3C_k}\! \left( \frac{\varepsilon}{C_{k,r}p^{k+1}}\right)^k \!& -\! \frac{1}{2}\! \frac{\varepsilon}{3C_k}\! \left( \frac{\varepsilon}{C_{k,r}p^{k+1}}\right)^k\!  -\! \frac{1}{4}\frac{\varepsilon}{3C_k}\! \left( \frac{\varepsilon}{C_{k,r}p^{k+1}}\right)^k \! >0\ .
\nonumber
\end{align}
We thus see that $x$ is not a global minimizer. \\
\indent This finishes the proof of Theorem \ref{thebigresult}. \hfill $\square$

\section{Proof of Theorem \ref{maintheorem2}}\label{destructionproofsection}
Even when the Aubry-Mather set $\mathcal{M}^{\omega}$ for a collection of local potentials $S_j$ is a Cantor set, these potentials may still admit a foliation of minimizers of rotation number $\omega$. In the language of Bangert \cite{bangert87}, this means that a lot of the minimizers in this foliation are ``nonrecurrent''. Nevertheless, as the following theorem shows, it is easy to remove these nonrecurrent minimizers by a smooth perturbation of the potentials. We remark that Theorem \ref{nofoliationtheorem} does not have an immediate holomorphic counterpart.

\begin{theorem}\label{nofoliationtheorem}
Assume that the $S_j$ are local potentials that satisfy conditions {\bf A}-{\bf E} of Section \ref{problemsetup} and let $\varepsilon>0$ be a perturbation parameter, $k\in \N_{\geq 2}$ a differentiability degree and $\omega\in\R \backslash \Q$ a rotation number. Assume moreover that the Aubry-Mather set $\mathcal{M}^{\omega}$ of rotation number $\omega$ for the local potentials $S_j$ is a Cantor set.
\\ \indent Then there exist local potentials $S^{\varepsilon}_j$ that satisfy conditions {\bf A}-{\bf E} of Section \ref{problemsetup} and the estimate 
$$||S_j^{\varepsilon} - S_j||_{C^k}\leq \varepsilon,$$
for which the Aubry-Mather set remains unchanged, i.e. $\mathcal{M}^{\omega, \varepsilon}=\mathcal{M}^{\omega}$, while at the same time the $S^{\varepsilon}_j$ admit no Birkhoff minimizers of rotation number $\omega$ outside this Aubry-Mather set.
\end{theorem}
\begin{proof}
We denote by $\Sigma^{\omega}$ the one-dimensional projection of $\mathcal{M}^{\omega}$:
$$\Sigma^{\omega}:=\{ x_0 \  |\ x\in \mathcal{M}^{\omega}\} \subset \R\ .$$
It is clear that $\Sigma^{\omega}$ is invariant under the integer translation $\xi\mapsto \xi+1$. Moreover, because the map $x\mapsto x_0$ from $\mathcal{M}^{\omega}$ into $\R$ is a homeomorphism, see \cite{moser86}, $\Sigma^{\omega}$ is a Cantor subset of $\R$. In particular, it is closed.  Therefore, there exists a $1$-periodic $C^{\infty}$ function $\phi:\R\to\R$ with $||\phi||_{C^k}\leq \varepsilon$ and
$$\phi(\xi)  \left\{ \begin{array}{ll} = 0 & \mbox{for}  \ \xi \in \Sigma^{\omega}, \\ > 0 & \mbox{for} \ \xi \notin \Sigma^{\omega}. \end{array} \right. 
 $$
We define the new potentials $S_j^{\varepsilon}$ by $S_j^{\varepsilon}(x):=S_j(x)+\phi(x_j)$, so that clearly $||S_j^{\varepsilon}-S_j||_{C^k}\leq \varepsilon$. We claim that the Aubry-Mather set $\mathcal{M}^{\omega, \varepsilon}$ of the perturbed potentials $S_j^{\varepsilon}$ equals the old Aubry-Mather set $\mathcal{M}^{\omega}$ of the unperturbed potentials $S_j$ and that whenever $x$ is a Birkhoff minimizer for the perturbed potentials, then $x\in \mathcal{M}^{\omega, \varepsilon} = \mathcal{M}^{\omega}$. \\
\indent 
To prove this claim, we first show that when $x\in \mathcal{M}^{\omega}$, then it is a minimizer for the perturbed potentials. Indeed, when $y:\Z\to \R$ is any sequence with finite support, then 
$$``W(x+y)-W(x)"=\sum_{j\in \Z} \left(S_j(x+y)-S_j(x)\right) +\sum_{j\in \Z} \left( \phi(x_j+y_j)-\phi(x_j) \right) \geq 0\ ,$$
because $x$ is a minimizer for the $S_j$, $\phi(x_j)=0$ for all $j$ and $\phi(x_j+y_j)\geq 0$ for all $j$. In particular, this shows that $\mathcal{M}^{\omega, \varepsilon} = \mathcal{M}^{\omega}$.\\
\indent Next, we let $x\in \mathcal{B}_{\omega}$ be a minimizer for the perturbed potentials. Let us assume that $x\notin \mathcal{M}^{\omega}$, i.e. that $x$ is ``nonrecurrent'' in the terminology of \cite{bangert87}. We will show that this leads to a contradiction.\\
\indent  By the results of Bangert \cite{bangert87} and our assumption that $\omega\in\R \backslash \Q$, we know that $x$ must then lie in a gap of $\mathcal{M}^{\omega}$. This means that there are $x^-, x^+\in \mathcal{M}^{\omega}$ so that $x^-\ll x \ll x^+$, but that there is no $y\in\mathcal{M}^{\omega}$ with $x^-\ll y \ll x^+$. \\
\indent By a standard result, see for instance \cite{bangert87}, \cite{moser86} or \cite{MramorRink1}, the gap $[x^-, x^+]:= \{ x^-\leq x \leq x^+\}$ is bounded in $l_1(\Z)$. More precisely, it holds that 
$$\sum_{j\in \Z} \left( x_j^+-x_j^- \right) \leq 1\ .$$ 
One can use this, cf. \cite{MramorRink1}, to prove that the function
$$W^{\varepsilon}_{[x^-,x^+]}: \ x\mapsto \sum_{j\in \Z} \left( S_j^{\varepsilon}(x)-S_j^{\varepsilon}(x^-)\right)$$ 
from $[x^-, x^+]$ into $\R$ is well-defined, absolutely convergent, nonnegative and Lipschitz-continuous. And that $x\in [x^-, x^+]$ is a global minimizer if and only if $W^{\varepsilon}_{[x^-, x^+]}(x)=0$. \\
\indent But clearly, when $x^-\ll x\ll x^+$, then
$$W^{\varepsilon}_{[x^-, x^+]}(x) =  \sum_{j\in \Z} \left( S_j(x)-S_j(x^-)\right) + \sum_{j\in \Z}\left(\phi(x_j)-\phi(x_j^-) \right) > 0\ .$$
This inequality holds because $x^-$ is a global minimizer for the $S_j$, because $\phi(x^-_j)=0$ and $\phi(x_j)>0$ for all $j$. This means that $x$ is not a global minimizer for the $S_j^{\varepsilon}$.
\end{proof}
As a consequence of Theorem \ref{nofoliationtheorem}, we obtain 
\begin{theorem}\label{corollaryofnofoliation}
Assume that the $S_j$ are local potentials that satisfy conditions {\bf A}-{\bf E} of Section \ref{problemsetup} and let $\varepsilon>0$ be a perturbation parameter, $k\in \N_{\geq 2}$ a differentiability degree, $\gamma>0$ and $\sigma > 2(1+k+k^2)$ real numbers and $\omega\in \mathcal{L}_{\gamma, \sigma}^-$ a rotation number.
\\ \indent Then there exist local potentials $\tilde S^{\varepsilon}_j$ that satisfy conditions {\bf A}-{\bf E} and the estimate 
$$||\tilde S_j^{\varepsilon} - S_j||_{C^k}\leq \varepsilon,$$
for which the Aubry-Mather set $\mathcal{M}^{\omega, \varepsilon}$ is a Cantor set. Moreover, the $\tilde S^{\varepsilon}_j$ admit no Birkhoff minimizers of rotation number $\omega$ outside this Aubry-Mather set.
\end{theorem} 
\begin{proof}
The local potentials $S_j^{\varepsilon}$ constructed in Theorem \ref{maintheoremprecise} satisfy $||S_j^{\varepsilon,2}-S_j||_{C^k}\leq \frac{2}{3}\varepsilon$. Their Aubry-Mather set $\mathcal{M}^{\omega, \varepsilon}$ is a Cantor set - and they already do not admit a minimal foliation. By Theorem \ref{nofoliationtheorem}, there exists now a further perturbation $\tilde S_{j}^{\varepsilon}$ of these $S_j^{\varepsilon}$, satisfying the estimate $||\tilde S_j^{\varepsilon}-S_j^{\varepsilon}||_{C^k} \leq\frac{1}{3}\varepsilon$, for which the conclusions of Theorem \ref{corollaryofnofoliation} hold. In particular, $||\tilde S_j^{\varepsilon}-S_j||_{C^k} \leq ||\tilde S_j^{\varepsilon}-S_j^{\varepsilon}||_{C^k} + ||S_j^{\varepsilon}-S_j||_{C^k} \leq \frac{2}{3}\varepsilon + \frac{1}{3}\varepsilon = \varepsilon$. 
\end{proof}
One can formulate a variant of Theorem \ref{corollaryofnofoliation} in the case that $\omega\in \mathcal{L}_{\gamma, \sigma}^+$.  This then proves Theorem \ref{maintheorem2} in the introduction. 

 \begin{small}
\bibliographystyle{amsplain}
\bibliography{destruction}
\end{small}

\end{document}